\documentclass[11pt]{article}

\usepackage[letterpaper,
            bindingoffset=0.2in,
            left=1in,
            right=1in,
            top=1in,
            bottom=1in,
            footskip=.25in]{geometry}

\usepackage{epsfig}
\usepackage{graphicx}
\usepackage{amsbsy}
\usepackage{amsmath}
\usepackage{amsfonts}
\usepackage{amssymb}
\usepackage{textcomp}
\usepackage[hidelinks]{hyperref}
\usepackage{aliascnt}

\newcommand{\mcm}[3]{\newcommand{#1}[#2]{{\ensuremath{#3}}}} 

\mcm{\tuple}{1}{\langle #1 \rangle}
\mcm{\name}{1}{\ulcorner #1 \urcorner}
\mcm{\Nbb}{0}{\mathbb{N}}
\mcm{\Zbb}{0}{\mathbb{Z}}
\mcm{\Rbb}{0}{\mathbb{R}}
\mcm{\Cbb}{0}{\mathbb{C}}
\mcm{\Qbb}{0}{\mathbb{Q}}
\mcm{\Acal}{0}{\cal A}
\mcm{\Bcal}{0}{\cal B}
\mcm{\Ccal}{0}{\cal C}
\mcm{\Dcal}{0}{\cal D}
\mcm{\Ecal}{0}{\cal E}
\mcm{\Fcal}{0}{\cal F}
\mcm{\Gcal}{0}{\cal G}
\mcm{\Hcal}{0}{\cal H}
\mcm{\Ical}{0}{\cal I}
\mcm{\Jcal}{0}{\cal J}
\mcm{\Kcal}{0}{\cal K}
\mcm{\Lcal}{0}{\cal L}
\mcm{\Mcal}{0}{\cal M}
\mcm{\Ncal}{0}{\cal N}
\mcm{\Ocal}{0}{{\cal O}}
\mcm{\Pcal}{0}{{\cal P}}
\mcm{\Qcal}{0}{{\cal Q}}
\mcm{\Rcal}{0}{{\cal R}}
\mcm{\Scal}{0}{{\cal S}}
\mcm{\Tcal}{0}{{\cal T}}
\mcm{\Ucal}{0}{{\cal U}}
\mcm{\Vcal}{0}{{\cal V}}
\mcm{\Wcal}{0}{{\cal W}}
\mcm{\Xcal}{0}{{\cal X}}
\mcm{\Ycal}{0}{{\cal Y}}
\mcm{\Zcal}{0}{{\cal Z}}
\mcm{\Mfrak}{0}{\mathfrak M}

\mcm{\restric}{0}{\upharpoonright}
\mcm{\upset}{0}{\uparrow}
\mcm{\onto}{0}{\twoheadrightarrow}
\mcm{\smallNbb}{0}{{\small \mathbb{N}}}
\DeclareMathOperator{\preop}{op}
\mcm{\op}{0}{^{\preop}}

{\begin{array}{c}
\setlength{\unitlength}{1em}}%
{\end{array}}

\usepackage{amsthm}

\newcommand{\theoremize}[2]{\newaliascnt{#1}{thm} \newtheorem{#1}[#1]{#2} \aliascntresetthe{#1}}

\theoremstyle{plain}
\newtheorem{thm}{Theorem}[section]
\theoremize{lem}{Lemma}
\theoremize{skolem}{Skolem}
\theoremize{fact}{Fact}
\theoremize{sublem}{Sublemma}
\theoremize{claim}{Claim}
\theoremize{subclaim}{Subclaim}
\theoremize{obs}{Observation}
\theoremize{prop}{Proposition}
\theoremize{cor}{Corollary}
\theoremize{que}{Question}
\theoremize{oque}{Open Question}
\theoremize{con}{Conjecture}

\theoremstyle{definition}
\theoremize{dfn}{Definition}
\theoremize{rem}{Remark}
\theoremize{eg}{Example}
\theoremize{exercise}{Exercise}
\theoremstyle{plain}

\newenvironment{claimproof}{
  \pushQED{\qed}%
  \trivlist
  \item[\hskip\labelsep
        \itshape
    Proof of Claim.]\ignorespaces
}{%
  
  \popQED\endtrivlist
}

\usepackage{verbatim}
\usepackage{enumerate}
\usepackage{enumitem}
\usepackage[all]{xy}

\usepackage{subfig}

\usepackage{authblk}

\title{Shallow Hitting Edge Sets in Uniform Hypergraphs\thanks{The results presented here are based on the Master's thesis of the first author~\cite{Planken-msc}.}}

\author[1]{Tim Planken}
\author[2]{Torsten Ueckerdt}

\affil[1]{University of Birmingham\\
    \href{mailto:txp265@student.bham.ac.uk}{txp265@student.bham.ac.uk}
}
\affil[2]{Karlsruhe Institute of Technology\\
    \href{mailto:torsten.ueckerdt@kit.edu}{torsten.ueckerdt@kit.edu}
}

\date{}

\DeclareMathOperator{\degree}{deg}

\DeclareMathOperator{\incident}{Inc}
\DeclareMathOperator{\Prob}{Pr}

\newcommand{\euler}{\mathrm{e}}

\usepackage{xr}

\begin{document}

\maketitle

\begin{abstract}
    A subset $M$ of the edges of a graph or hypergraph is \emph{hitting} if $M$ covers each vertex of $H$ at least once, and $M$ is \emph{$t$-shallow} if it covers each vertex of $H$ at most $t$ times.
    We consider the existence of shallow hitting edge sets and the maximum size of shallow edge sets in $r$-uniform hypergraph $H$ that are regular or have a large minimum degree.
    Specifically, we show the following.
    \begin{itemize}
        \item Every $r$-uniform regular hypergraph has a $t$-shallow hitting edge set with $t = O(r)$.
        \item Every $r$-uniform regular hypergraph with $n$ vertices has a $t$-shallow edge set of size $\Omega(nt/r^{1+1/t})$.
        \item Every $r$-uniform hypergraph with $n$ vertices and minimum degree $\delta_{r-1}(H) \geq n/((r-1)t+1)$ has a $t$-shallow hitting edge set.
        \item Every $r$-uniform $r$-partite hypergraph with $n$ vertices in each part and minimum degree $\delta'_{r-1}(H) \geq n/((r-1)t+1) +1$ has a $t$-shallow hitting edge set.
    \end{itemize}
    We complement our results with constructions of $r$-uniform hypergraphs that show that most of our obtained bounds are best-possible.
\end{abstract}

\section{Introduction}

Two of the most fundamental results in graph theory are the necessary and sufficient conditions for the existence of perfect matchings in bipartite graphs due to Hall's ``marriage theorem''~\cite{Hall-1935} and in general graphs due to Tutte's ``odd component count''~\cite{Tutte-1950}.
However, the situation is less understood and also more complicated in hypergraphs.
For example, while we know how to detect perfect matchings in graphs in polynomial time~\cite{Edmonds-1965,MV-algorithm-1980} and in bipartite graphs even in near-linear time~\cite{near-linear-flow-2022}, the problem is NP-complete already in $3$-uniform $3$-partite hypergraphs~\cite{Karp-1972}, indicating that there is (in all likelihood) no nice necessary and sufficient criterium for perfect matchings in hypergraphs.

A simple application of Hall's marriage theorem shows that every regular (all vertices have the same degree) bipartite graph admits a perfect matching.
However, this does not hold for hypergraphs, as for example illustrated by the $2$-regular hypergraph in Figure~\ref{fig:3-graph} which admits no perfect matching.
For some integer $r \geq 2$, a hypergraph $H = (V,E)$ is $r$-uniform if every edge $e \in E$ consists of exactly $r$ vertices.
That is, $2$-uniform hypergraphs are just graphs, while the hypergraph in Figure~\ref{fig:3-graph} is $3$-uniform.
Moreover, $H = (V,E)$ is $r$-partite if its vertices can be partitioned into $r$ parts such that each edge contains exactly one vertex of each part.
That is, $2$-uniform $2$-partite hypergraphs are exactly the bipartite graphs, while the hypergraph in Figure~\ref{fig:3-graph} is $3$-partite with the a $3$-partition indicated by the shapes of the vertices.
Evidently, as soon as $r \geq 3$, not all regular $r$-uniform $r$-partite hypergraphs admit perfect matchings.

\begin{figure}
    \centering
    \includegraphics{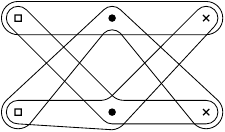}
    \caption{A $2$-regular $3$-uniform $3$-partite hypergraph with no perfect matching.}
    \label{fig:3-graph}
\end{figure}

We are concerned with a relaxation of perfect matchings, called shallow hitting edge sets.
While $M$ is a perfect matching if every vertex has exactly one incident edge in $M$, an edge subset $M$ is a \emph{$t$-shallow hitting edge set} for a positive integer $t$ if every vertex has at least one incident edge in $M$ (this is the ``hitting'' part) and at most $t$ incident edges in $M$ (this is the ``$t$-shallow'' part).
For example, any set of three or four edges in the hypergraph $H$ in Figure~\ref{fig:3-graph} is a $2$-shallow hitting edge set of $H$.
In this paper, we shall present sufficient conditions for $t$ and $r$ for which every regular $r$-uniform hypergraph admits a $t$-shallow hitting edge set.

Of course, if a graph or hypergraph admits no perfect matching, it is interesting to prove lower bounds on the size of its largest matchings, for example for regular graphs~\cite{HenYeo-max-matching-2007}.
In this paper, we shall present for regular $n$-vertex $r$-uniform hypergraphs $H$ lower bounds for the size of their maximum $t$-shallow edge sets in terms of $n$, $r$ and $t$.

Another simple application of Hall's marriage theorem shows that every bipartite graph with $n$ vertices in each part and minimum degree at least $n/2$ admits a perfect matching.
R\"odl, Ruci\'nski and Szemer\'edi~\cite{rodl2009perfect}, see also K\"uhn and Osthus~\cite{kuhnmatchings}, extend this by showing that every $n$-vertex $r$-uniform hypergraph $H$ with large enough $\delta_{r-1}(H)$ in terms of $n$ and $r$ (see Section~\ref{sec:minimum-degree-conditions} for the definition) admits a perfect matching.
Similarly, every $r$-uniform $r$-partite hypergraph with parts of size $n$ and large enough $\delta'_{r-1}(H)$ in terms of $n$ and $r$ admits a perfect matching~\cite{kuhnmatchings,aharoniperfectmatchings}.
In this paper, we extend these results to $t$-shallow hitting edge sets.

\subsection{Our results and organization of the paper}

All hypergraphs considered in this paper are uniform and finite, but may have parallel edges.
We give the formal definitions of terms such as $r$-uniform, $r$-partite, girth, near-regular, $t$-shallow edge sets, hitting edge sets, etc.~in Section~\ref{subsec:preliminaries}.

\medskip

In Section~\ref{sec:shallow-hitting} we investigate the existence of \textbf{shallow hitting edge sets in regular uniform hypergraphs}.
We seek to prove that every $r$-uniform regular hypergraph admits a $t$-shallow hitting edge set with $t$ being as small as possible with respect to $r$.
On the positive side, Theorem~\ref{theorem:upper-bound-uniform-regular-hypergraph} in Section~\ref{subsec:upper-bound-general} gives that every $r$-uniform regular hypergraph has a $t$-shallow hitting edge set with $t = O(r)$.
In Theorem~\ref{thm:upper-bound-high-girth} in Section~\ref{subsec:upper-bound-girth} we improve the upper bound on $t$ to $t = O(\ln r)$ for the special case of $r$-uniform regular hypergraphs whose girth is at least four.
We complement this with lower bounds in Section~\ref{subsec:lower-bound-shallow-hitting}.
We present in Theorem~\ref{theorem:01-truncated-projective-space} a construction of an $r$-uniform regular hypergraph that admits $t$-shallow hitting edge sets only for $t = \Omega(\log r)$.

In fact, we prove our upper bounds in Sections~\ref{subsec:upper-bound-general} and~\ref{subsec:upper-bound-girth} for the more general case of $r$-uniform near-regular hypergraphs.
In a near-regular hypergraph we only require that the quotient $\Delta/\delta$ of the maximum degree over the minimum degree is bounded by some constant $\mu = O(1)$.
On the other hand, the lower bounds in Section~\ref{subsec:lower-bound-shallow-hitting} are provided in the more restricted case of $r$-uniform regular hypergraphs, and which are additionally restricted to be $r$-partite.

\medskip
In Section~\ref{sec:maximum-shallow-edge-sets} we investigate the \textbf{maximum size of shallow edge sets in regular uniform hypergraphs}.
Again, we consider $r$-uniform regular hypergraphs $H = (V,E)$.
But here we seek to find $t$-shallow edge sets $M \subseteq E$ that are not necessarily hitting.
Of course, $M = \emptyset$ is always a valid choice, while the challenge here is to have $|M|$ as large as possible in terms of $t$, $r$ and $n = |V|$.
On the positive side, Theorem~\ref{theorem:03-01-lower-bound} in Section~\ref{subsec:03-01-lower-bound} gives that every $r$-uniform regular hypergraph with $n$ vertices admits a $t$-shallow edge set $M$ of size
\[
    |M| = \Omega\left(\frac{nt}{r^{1+1/t}}\right).
\]
We complement this with an asymptotically matching upper bound in Section~\ref{subsec:04-03-combinatorial-designs}.
We present in Theorem~\ref{theorem:04-03-upper-bound} a construction of $r$-uniform regular hypergraphs with $n$ vertices that admit $t$-shallow edge sets $M$ only of size 
\[
    |M| = O\left(\frac{nt}{r^{1+1/t}}\right).
\]
In fact, we again prove our lower bound in Section~\ref{subsec:03-01-lower-bound} for the more general case of near-regular hypergraphs, while our upper bound in Section~\ref{subsec:04-03-combinatorial-designs} is provided in the more restricted case of regular $r$-uniform $r$-partite hypergraphs.

\medskip
In Section~\ref{sec:minimum-degree-conditions} we investigate the existence of \textbf{shallow hitting edge sets in uniform hypergraphs of large minimum degree}.
The famous theorem of Dirac~\cite{dirac1952some} states that every graph on $n$ vertices with minimum degree at least $n/2$ contains a Hamiltonian cycle, while some graphs with minimum degree less than $n/2$ do not.
We consider the Dirac-type problem of finding a minimum degree condition for uniform hypergraphs to contain a shallow hitting edge set.
We focus on two of the several possibilities to define the minimum degree of $r$-uniform hypergraphs with $r > 2$.

In Section~\ref{subsec:uniform-hypergraphs} we consider $r$-uniform hypergraphs $H$ and define $\delta_{r-1}(H)$ to be the smallest integer $d$ such that every $(r-1)$-set of vertices of $H$ forms an edge with $d$ distinct vertices in $H$.
We prove in Theorem~\ref{theorem:02-09-large-degree-uniform-hypergraph} that every $r$-uniform hypergraph $H$ with $n = |V|$ vertices has a $t$-shallow hitting edge set, $t \geq 2$, provided
\[
    \delta_{r-1}(H) \geq \frac{n}{(r-1)t+1}.
\]
We complement this with a construction in Theorem~\ref{theorem:02-09-uniform-lower-bound} of $r$-uniform hypergraphs with $n$ vertices and
\[
    \delta_{r-1}(H) \geq \frac{n}{(r-1)t+1}-1
\]
that have no $t$-shallow hitting edge sets.

In Section~\ref{subsec:uniform-partite-hypergraphs} we consider $r$-uniform $r$-partite hypergraphs $H$ and define $\delta'_{r-1}(H)$ to be the smallest integer $d$ such that every $(r-1)$-set of vertices from pairwise distinct parts of $H$ forms an edge with $d$ distinct vertices of the remaining part in $H$.
We prove in Theorem~\ref{theorem:02-09-upper-bound} that every $r$-uniform $r$-partite hypergraph with parts of size $n$ has a $t$-shallow hitting edge set, $t \geq 2$, provided
\[
    \delta'_{r-1}(H) \geq \frac{n}{(r-1)t+1} + 1.
\]
We complement this with a construction in Theorem~\ref{theorem:02-09-lower-bound} of $r$-uniform $r$-partite hypergraphs with parts of size $n$ and
\[
    \delta'_{r-1}(H) \geq \frac{n}{(r-1)t+1}-1
\]
that have no $t$-shallow hitting edge sets.

\medskip
We end this paper with a brief discussion in Section~\ref{sec:conclusions}.

\subsection{Related work}
\label{subsec:related-work}

Let us just briefly mention some further related work on matchings in hypergraphs.
Until now, there is no generalization of Hall’s Theorem or Tutte’s Theorem to hypergraphs.
There is an equivalence for the existence of perfect matchings in so-called balanced hypergraphs, see~\cite{CCKV-balanced-1996} and~\cite{HucTri-balanced-2002}.
On the other hand, there are two ways to extend the sufficient condition of Hall’s Theorem to uniform bipartite hypergraphs (here, a hypergraph $H=(V,E)$ is bipartite if there exists a partition $V = A\dot\cup B$ such that every edge has exactly one vertex in $A$).
Haxell~\cite{Haxell-1995} gave a sufficient condition in terms of minimum vertex covers while Aharoni and Haxell~\cite{AhaHax-2000} gave a sufficient condition in terms of maximum matchings. 
Both theorems are equivalent to Hall’s Theorem in the case of bipartite graphs, but do not give a necessary condition in general.

The concept of $t$-shallow hitting edge sets has been considered for graphs implicitly under several names.
Specifically, there is a lot of research on subgraphs $G'$ of a given graph $G$ such that for each vertex $v$ its number of incident edges in $G'$ lies in some predefined set $S_v$, as introduced by Lov\'asz~\cite{Lovasz1972_Factorization}.
(It would be $S_v = \{1,\ldots,t\}$ for $t$-shallow hitting edge sets.)
Let us just mention that Berge and Las Vergnas~\cite{berge1978existence} give a Hall-type characterization for the case of bipartite graphs, as well as a Tutte-type characterization for general graphs.

Finally, hitting edge sets of a hypergraph $H$ are equivalent to ``hitting vertex sets'' of the dual hypergraph $H^*$ of $H$.
Such hitting vertex sets are also known as vertex covers, transversals, or hitting sets in the literature.
Shallow hitting (vertex) sets are for example a standard tool for coloring geometric hypergraphs~\cite{KesDom-shallow-definition-2019}.

\subsection{Preliminaries}
\label{subsec:preliminaries}

A hypergraph is a tuple $H=(V,E)$ consisting of a finite set $V$ of vertices and a finite multiset $E$ of edges\footnote{Sometimes $e$ is called a \emph{hyper}edge if $|e| \geq 3$, but let us call $e$ simply an \emph{edge}, independent of its size $|e|$.}, where each edge $e \in E$ is a subset of $V$.
For a vertex $v \in V$, we denote the set of edges incident to $v$ in $H$ by $\incident(v)$.
For $M \subseteq E$ and $v \in V$, the \emph{degree $\degree_M(v)$ of $v$ in $M$} is the number of incident edges at $v$ that are in $M$, i.e., $\degree_M(v)=|\incident(v) \cap M|$.
The \emph{degree $\degree(v)$ of $v$} is $\degree(v)=\degree_E(v) = |\incident(v)|$.

The maximum (respectively minimum) degree of $H$ is the largest (respectively smallest) degree of a vertex in $H$.
A hypergraph $H$ with maximum degree $\Delta$ and minimum degree $\delta$ is \emph{regular} if $\Delta = \delta$, and \emph{$\mu$-near-regular} for a real number $\mu \in \Rbb_{\geq 0}$ if $\delta > 0$ and $\Delta / \delta \leq \mu$.
Thus, a hypergraph is $1$-near-regular if and only if it is regular and has minimum degree at least $1$.
Of course, every hypergraph with minimum degree $\delta > 0$ is $\mu$-near-regular for every large enough $\mu$.
When we later consider ``near-regular hypergraphs'', this refers to the class of all $\mu$-near-regular hypergraphs for some fixed $\mu \in \Rbb_{\geq 0}$.

For an integer $r \geq 1$, a hypergraph $H = (V,E)$ is \emph{$r$-uniform} if we have $|e| = r$ for each $e \in E$.
That is, $2$-uniform hypergraphs are graphs (with no loops but possibly parallel edges).
In the literature, $r$-uniform hypergraphs are sometimes simply called \emph{$r$-graphs}.

A hypergraph $H$ is called \emph{$r$-partite} if there exists a partition $V_1,\ldots,V_r$ of $V$ into $r$ pairwise disjoint parts such that for each $e \in E$ and $i \in \{1,\ldots,r\}$ we have $|e \cap V_i| \leq 1$.
Whenever we consider an $r$-partite hypergraph $H = (V,E)$, we fix one such partition $V = V_1 \dot\cup \cdots \dot\cup V_r$ and call each $V_i$, $i=1,\ldots,r$, a \emph{part} of $H$.
The smallest $r$ for which $H$ is $r$-partite is the \emph{chromatic number} of $H$, but we shall not use this notion here.
Arguably, $r$-uniform $r$-partite hypergraphs are the most natural generalization of bipartite graphs to larger uniformity.

We say that $H^*=(V^*,E^*)$ is the \emph{dual hypergraph} of $H$ if $V^*=E$ is the set of vertices and $E^* = \{\incident(v) \mid v \in V\}$ is the set of edges.

\smallskip

Let $H = (V,E)$ be a hypergraph.
For an edge set $M \subseteq E$, a vertex $v \in V$ is said to be \emph{covered by $M$} if $\degree_M(v) \geq 1$.
An edge set $M \subseteq E$ is \emph{hitting} if every vertex in $H$ is covered by $M$.
For an integer $t \geq 0$, an edge set $M \subseteq E$ is \emph{$t$-shallow} if $\degree_M(v) \leq t$ for each vertex $v$.

Finally, recall that $o_r(1)$ denotes the class of all functions $f \colon \Rbb_{\geq 0} \to \Rbb_{\geq 0}$ for which $f(r) \to 0$ as $r \to \infty$.
At several places in this paper we seek to describe the asymptotic behavior of a function in terms of \emph{two} parameters: $r$ and $\mu$.
To this end, let $o_{\mu r}(1)$ be the class of all functions $f \colon \Rbb_{\geq 0}^2 \to \Rbb_{\geq 0}$ for which $f(\mu,r) \to 0$ whenever $\mu r \to \infty$.

\section{Shallow Hitting Edge Sets in Uniform Near-Regular Hypergraphs}
\label{sec:shallow-hitting}

\subsection{Upper bound for the general case}
\label{subsec:upper-bound-general}

Let $A_1, \dots, A_n$ be events.
A graph $G=(V,E)$ is a \emph{dependency graph} of these events if $V=\{A_1, \dots, A_n\}$ and the following holds for all $i$.
For each event $A_i$ and each set $S$ of non-adjacent events with $A_i \notin S$ it holds that $\Prob\left[A_i \mid \bigwedge_{A_j \in S} \overline A_j\right] \leq \Prob[A_i]$.
The \emph{degree of dependence} of the events $A_1, \dots, A_n$ is the smallest possible maximum degree of a dependency graph.
If each ``bad'' event $A_i$ only occurs with small probability and each bad event is dependent of a small amount of other bad events, the Lovász Local Lemma is a tool to prove that it is possible that no bad event occurs.

\begin{lem}[Lovász Local Lemma~\cite{spencer1977asymptotic}]
    \label{lem:lovasz-local-lemma}
    Let $A_1, \dots, A_n$ be events with $\Prob[A_i] \leq p$ and let $d$ be their degree of dependence.
    If $\euler p (d + 1) \leq 1$, then $\Prob[\overline A_1 \overline A_2 \cdots \overline A_n] > 0$.
\end{lem}
\begin{rem}
\label{rem:weak-dependency-graph}
    In the original Lovász Local Lemma, the notion of a dependency graph is replaced by a slightly stronger definition in which each event $A_i$ must be mutually independent of all non-adjacent events $A_j$ with $i \neq j$.
    It can be seen by the proof in~\cite{alon2016probabilistic} that the Lovász Local Lemma remains true if we use the weaker notion of a dependency graph as defined above.
    In Lemma~\ref{lemma:upper-bound-uniform-regular-hypergraph}, we will use the original version of the Lovász Local Lemma (which is included as a special case in Lemma~\ref{lem:lovasz-local-lemma}).
    In Lemma~\ref{lemma:upper-bound-high-girth}, we need the weaker notion of a dependency graph and therefore the slightly stronger Lovász Local Lemma as in Lemma~\ref{lem:lovasz-local-lemma}.
\end{rem}

In the next lemma, we want to find shallow hitting edge sets in uniform near-regular hypergraphs.
For this, we use a random experiment that generates a hitting edge set $M$.
Then, we only need to ensure the condition that $M$ is also shallow which will be done using the Lovász Local Lemma.

\begin{lem}
    \label{lemma:upper-bound-uniform-regular-hypergraph}
    Let $H=(V,E)$ be an $r$-uniform $\mu$-near-regular hypergraph and let $t$ be a positive integer satisfying
    \begin{equation}
        \label{eq:01-linear-upper-bound}
        \frac{t!}{t+1} \geq \euler \mu^{t+1} r^{t+3}~.
    \end{equation}
    Then, $H$ has a $t$-shallow hitting edge set.
\end{lem}

\begin{proof}
    Let $\Delta$ be the maximum degree and $\delta$ be the minimum degree of $H$.
    We build an edge set $M \subseteq E$ with the following random experiment.
    For each vertex $v \in V$, we pick an incident edge $e \in \incident(v)$ uniformly at random and add it to the edge set $M$.
    Clearly, $M$ is a hitting edge set.
    Observe that for $t \geq \Delta$, $M$ is clearly $t$-shallow.
    Thereby, assume that $t < \Delta$.
    We use the Lovász Local Lemma to prove that there exists an edge set $M$ that also satisfies $\degree_M(v) \leq t$ for all vertices $v \in V$.
    For a set $F$ of edges, denote by $V(F)$ the set of vertices $\bigcup_{e \in F} \, e$.
    
    Define the set $\mathcal{F}$ to be the set of all edge sets $F \subseteq E$ of size $t+1$ such that there exists a vertex $v \in V$ with $F \subseteq \incident(v)$, i.e.,
    \begin{equation*}
        \mathcal{F} = \{F \subseteq E \colon |F|=t+1, F \subseteq \incident(v) \text{ for some }  v \in V\}~.
    \end{equation*}
    For a set $F \in \mathcal{F}$, we denote the event that $F \subseteq M$ by $A_F$.
    These are the bad events used in the Lovász Local Lemma.
    
    To apply the Lovász Local Lemma, we have to find a bound $p$ such that $\Prob[A_F] \leq p$ holds for all $F \in \mathcal{F}$.
    For a vertex $v \in V$, we denote by $P_v$ the random variable with range $\incident(v)$ that describes which edge is picked at the vertex $v$.
    Moreover, for $e \in E$ and $v \in e$, we define $B_{e,v}$ to be the event that $P_v = e$ and $B_e$ to be the event that $e \in M$.
    Since every vertex has degree at least $\delta$, it holds that $\Prob[B_{e,v}] = \Prob[P_v = e] \leq 1/\delta$ for all $e \in E$ and $v \in e$.
    For an edge $e \in E$, we can now bound the probability for the event $B_e$ by
    \begin{equation*}
        \Prob[B_e] 
        = \Prob\left[\bigvee_{v \in e} B_{e,v}\right] 
        \leq \sum_{v \in e} \Prob[B_{e,v}] 
        \leq \frac{r}{\delta}~.
    \end{equation*}
    For $e \in E$ and $F \subseteq E$ with $e \notin F$, it holds that $\Prob[B_{e} \mid \cap_{e' \in F} B_{e'}] \leq \Prob[B_{e}]$ and hence $\Prob[B_{e} \cap \left(\cap_{e' \in F} B_{e'}\right)]\leq\Prob[B_{e}] \cdot \Prob[\cap_{e' \in F} B_{e'}]$.
    Thus, for each $F \in \mathcal{F}$, the probability that the event $A_F$ occurs can be bounded by
    \begin{equation*}
        \Prob[A_F]
        = \Prob\left[ \bigwedge_{e \in F} B_e \right]
        \leq \prod_{e \in F} \Prob[B_e]
        \leq \left( \frac{r}{\delta} \right)^{t+1}~.
    \end{equation*}
    
    We construct a dependency graph $G_\text{D}=(V_\text{D}, E_\text{D})$ with $V_\text{D}=\{A_F \mid F \in \mathcal{F}\}$.
    Two events $A_F$ and $A_{F'}$ are adjacent in $G_\text{D}$ if and only if $V(F) \cap V(F') \neq \emptyset$.
    Note that the event $B_e$ is mutually independent of all events $B_{e'}$ with $e \cap e' = \emptyset$.
    Thus, each event $A_F$ for $F \in \mathcal{F}$ is mutually independent of all non-adjacent events $A_{F'}$.
    
    In the next step, we have to bound the degree of dependence of these events.
    For an edge $e \in E$ and a vertex $v \in e$, denote by $\mathcal{F}_{e,v} \subseteq \mathcal{F}$ the set of edge sets $F \in \mathcal{F}$ with $e \in F$ and $F \subseteq \incident(v)$.
    We bound the size of the sets $\mathcal{F}_{e,v}$.
    The edge $e$ is contained in all sets $F \in \mathcal{F}_{e,v}$ and the vertex $v$ has degree at most $\Delta$.
    Since every edge set $F \in \mathcal{F}$ has size exactly $t+1$, there are at most $\binom{\Delta - 1}{t}$ ways to choose the remaining $t$ edges (which must be incident with $v$).
    Thus,
    \begin{equation*}
        |\mathcal{F}_{e,v}| \leq \binom{\Delta - 1}{t} \quad \text{for all } e \in E \text{ and } v \in e~.
    \end{equation*}
    Let $F \in \mathcal{F}$ be an arbitrary but fixed set of edges.
    In order to bound the degree of dependence, we count the number of edge sets $F' \in \mathcal{F}$ with $V(F) \cap V(F') \neq \emptyset$.
    Let $F' \in \mathcal{F}$ be an arbitrary edge set with $A_{F'}$ adjacent to $A_F$ in the dependency graph $G_\text{D}$. 
    Since $V(F) \cap V(F') \neq \emptyset$, there must exist a vertex $\tilde v$ in $V(F) \cap V(F')$ with an edge $e'$ in $F'$ which is incident to $\tilde v$.
    By definition of $\mathcal{F}$, there is a vertex $v' \in V$ with $F \subseteq \incident(v')$.
    As $e' \in F'$, this vertex must be incident to $e'$, i.e., $v' \in e'$.
    Thus, $F'$ is in the set $\mathcal{F}_{e', v'}$.
    Since the set $F$ itself fulfills all the described properties, it is also in a set $\mathcal{F}_{e,v}$ and we count it too.
    We can now bound the degree of dependence $d$ by
    \begin{equation*}
        d+1 \leq \max_{F \in \mathcal F}\left|\bigcup\limits_{\tilde v \in V(F)} \bigcup\limits_{e' \in \incident(\tilde v)} \bigcup\limits_{v' \in e'} \mathcal{F}_{e',v'}\right|
        \leq (t+1)r \cdot \Delta \cdot r \cdot \binom{\Delta - 1}{t}~.
    \end{equation*}
    Since we assumed $\Delta - 1 \geq t$, it holds that
    \begin{equation*}
        \binom{\Delta - 1}{t} \leq \frac{(\Delta - 1)^{t}}{t!} \leq \frac{\Delta^{t}}{t!}
    \end{equation*}
    and thus
    \begin{equation}
        d + 1 \leq \frac{t+1}{t!} r^2 \Delta^{t+1}~.
    \end{equation}
    To apply the Lovász Local Lemma, we calculate $\euler p (d 
 + 1)$.
    With Equation~\eqref{eq:01-linear-upper-bound}, it follows that
    \begin{equation*}
        \euler p (d+1) \leq \euler \cdot \left( \frac{r}{\delta}\right)^{t+1} \cdot \frac{t+1}{t!} r^2 \Delta^{t+1}
        \leq \euler \frac{t+1}{t!} \mu^{t+1} r^{t+3} 
        \leq 1~.
    \end{equation*}
    
    By the Lovász Local Lemma, the probability that no event $A_F$ with $F \in \mathcal{F}$ occurs is greater than zero.
    Thus, there exists an edge set $M \subseteq E$ such that every vertex has degree $1 \leq \degree_M(v) \leq t$ in $M$.
\end{proof}

\begin{thm}
    \label{theorem:upper-bound-uniform-regular-hypergraph}
    Every $r$-uniform $\mu$-near-regular hypergraph has a $t(\mu, r)$-shallow hitting edge set, where $t(\mu, r) = \euler \mu r (1+o_{\mu r}(1))$.
\end{thm}

\begin{proof}
    We show that 
    \[
        t=1+\left\lceil \euler \mu r \left(\euler^2 \mu^2 r^4 \right)^{1/(\euler \mu r)} \right\rceil = \euler \mu r (1+o_{\mu r}(1))
    \]
    satisfies Equation~\eqref{eq:01-linear-upper-bound}.
    Indeed, with Stirling's Formula $t! \geq \sqrt{2 \pi t} (t/\euler)^t$ and with $(t+1)/\sqrt{2 \pi t} \leq t$ for $t \geq 1$ it follows that
    \begin{equation*}
        \begin{split}
            \euler \frac{t+1}{t!} \mu^{t+1} r^{t+3} 
            &\leq \euler \frac{t+1}{\sqrt{2 \pi t}} \mu r^3 \cdot \left( \frac{\euler \mu r}{t} \right)^t
            \leq \euler^2 \mu^2 r^4 \left( \frac{\euler \mu r}{t-1} \right)^{t-1}\\
            &\leq \euler^2 \mu^2 r^4 \left( \euler^2 \mu^2 r^4 \right)^{-(\euler^2 \mu^2 r^4)^{1/(\euler \mu r)}} \leq 1~.
        \end{split}
    \end{equation*}
\end{proof}

\subsection{Upper bound for girth at least four}
\label{subsec:upper-bound-girth}
In the next theorem, we show a better bound for uniform near-regular hypergraphs without small cycles.
Let $H=(V,E)$ be a hypergraph.
A \emph{cycle of length $k$} is a sequence $v_1 e_1 v_2 e_2 \dots v_k e_k v_{k+1}$ of distinct vertices $v_1,v_2,\dots,v_k$ and distinct edges $e_1,e_2,\dots,e_k$ such that $v_1 = v_{k+1}$ and $\{v_i,v_{i+1}\}\subseteq e_i$ for $i=1,2,\dots, k$.
The \emph{girth} of a hypergraph $H$ is the smallest possible length of a cycle in $H$.
In Lemma~\ref{lemma:upper-bound-high-girth} and Theorem~\ref{thm:upper-bound-high-girth}, we prove a better bound for uniform regular hypergraphs of girth at least $4$.

For that, we need the following lemma.
It says that if we have two events $A$ and $B$ that are independent by conditioning on the set $\{\tilde C_1,\dots, \tilde C_n\}$ of events for all $\tilde C_i \in \{C_i, \overline C_i\}$, $i=1,\dots,n$, and each event $C_i$ decreases the probability that $A$ occurs but raises the probability that $B$ occurs, then $\Prob[A \mid B] \leq \Prob[A]$.
Moreover, we use the abbreviation
\[
    \sum_{\substack{\tilde C_i \in \{C_i, \overline C_i\} \\ i=1,\dots,k}} f(\tilde C_1, \dots, \tilde C_k)
    = \sum_{\substack{\tilde C_i \in \{C_i, \overline C_i\} \\ i=1,\dots,k-1}} \; \sum_{\tilde C_k \in \{C_k, \overline C_k\}} f(\tilde C_1, \dots, \tilde C_k)
\]
for $k \geq 1$ and with 
\[
    \sum_{\substack{\tilde C_i \in \{C_i, \overline C_i\} \\ i\in \emptyset}} a = a~.
\]
\begin{lem}
    \label{lemma:helper-lemma-triangle-free-probabilities}
    Let $A$, $B$ and $C_1,\dots,C_n$ be events such that the following holds for all events $\tilde C_i \in \{C_i, \overline C_i\}$, $i=1,\dots,n$,
    \begin{enumerate}
        \item\label{itm:helper-lemma-eq01}
        $\Prob\left[A \mid B, \bigwedge_{i=1}^n \tilde C_i\right] = \Prob\left[A \mid \bigwedge_{i=1}^n \tilde C_i\right]$,
        \item\label{itm:helper-lemma-eq02}
        $\Prob\left[A \mid C_{k+1}, \bigwedge_{i=1}^k \tilde C_i\right] \leq \Prob\left[ A \mid \bigwedge_{i=1}^k \tilde C_i\right]$ for all $0 \leq k < n$,
        \item\label{itm:helper-lemma-eq03}
        $\Prob\left[ B \mid C_{k+1}, \bigwedge_{i=1}^k \tilde C_i\right] \geq \Prob\left[ B \mid \bigwedge_{i=1}^k \tilde C_i\right]$ for all $0 \leq k < n$.
    \end{enumerate}
    Then, $\Prob[A \mid B] \leq \Prob[A]$.
\end{lem}
\begin{proof}
    We show that the inequality 
    \begin{equation}
        \label{eq:helper-lemma-high-girth}
        \Prob\left[ A \mid B \right] 
        \leq \sum_{\substack{\tilde C_i \in \{C_i, \overline C_i\} \\ i=1,\dots,k}} \Prob\left[ \bigwedge_{i=1}^{k} \tilde C_i \mid B \right] \Prob\left[ A \mid \bigwedge_{i=1}^{k} \tilde C_i \right]
    \end{equation}
    holds for all $k=0,1,\dots,n$ by induction on $n-k$.
    Clearly, for $k = 0$, it follows that $\Prob[A \mid B] \leq \Prob[A]$.
    For the base case $k=n$ it holds that
    \[
        \Prob\left[ A \mid B \right] 
        = \sum_{\substack{\tilde C_i \in \{C_i, \overline C_i\} \\ i=1,\dots,n}} \Prob\left[ \bigwedge_{i=1}^{n} \tilde C_i \mid B \right] \Prob\left[ A \mid B, \bigwedge_{i=1}^{n} \tilde C_i \right]~.
    \]
    With Condition~\ref{itm:helper-lemma-eq01} of Lemma~\ref{lemma:helper-lemma-triangle-free-probabilities}, the base case $k=n$ of Equation~\eqref{eq:helper-lemma-high-girth} follows.
    Now, suppose that $0 \leq k < n$ and that the inequality holds for $k+1$.
    Then, we use the induction hypothesis to get
    \begin{equation*}
        \begin{split}
            \Prob\left[ A \mid B \right]
            &\leq \sum_{\substack{\tilde C_i \in \{C_i, \overline C_i\} \\ i=1,\dots,k+1}} \Prob\left[ \bigwedge_{i=1}^{k+1} \tilde C_i \mid B \right] \Prob\left[ A \mid \bigwedge_{i=1}^{k+1} \tilde C_i \right]
            = \sum_{\substack{\tilde C_i \in \{C_i, \overline C_i\} \\ i=1,\dots,k}} \Prob\left[ \bigwedge_{i=1}^{k} \tilde C_i \mid B \right] \cdot X
        \end{split}
    \end{equation*}
    with
    \begin{equation*}
        \begin{alignedat}{2}
             X &= &&
             \Prob\left[ C_{k+1} \mid B, \bigwedge_{i=1}^k \tilde C_i \right] 
             \Prob\left[ A \mid C_{k+1}, \bigwedge_{i=1}^k \tilde C_i \right] + 
             \Prob\left[ \overline C_{k+1} \mid B, \bigwedge_{i=1}^k \tilde C_i \right] 
             \Prob\left[ A \mid \overline C_{k+1}, \bigwedge_{i=1}^k \tilde C_i \right] \\
             &= &&\Prob\left[ C_{k+1} \mid B, \bigwedge_{i=1}^k \tilde C_i \right] 
             \underbrace{\left( \Prob\left[ A \mid C_{k+1}, \bigwedge_{i=1}^k \tilde C_i \right] - \Prob\left[ A \mid \overline C_{k+1}, \bigwedge_{i=1}^k \tilde C_i \right] \right)}_{\leq 0}\\ 
             & &&+ 
             \Prob\left[ A \mid \overline C_{k+1}, \bigwedge_{i=1}^k \tilde C_i \right]~.
        \end{alignedat}
    \end{equation*}
    It follows from Condition~\ref{itm:helper-lemma-eq02} that $\Prob\left[ A \mid C_{k+1}, \bigwedge_{i=1}^k \tilde C_i \right] \leq \Prob\left[ A \mid \overline C_{k+1}, \bigwedge_{i=1}^k \tilde C_i \right]$ and it follows from Condition~\ref{itm:helper-lemma-eq03} that $\Prob\left[ C_{k+1} \mid B, \bigwedge_{i=1}^k \tilde C_i \right] \geq \Prob\left[ C_{k+1} \mid \bigwedge_{i=1}^k \tilde C_i \right]$.
    We use the second inequality and get
    \begin{equation*}
        \begin{alignedat}{2}
            X &\leq && \Prob\left[ C_{k+1} \mid \bigwedge_{i=1}^k \tilde C_i \right] 
            \Prob\left[ A \mid C_{k+1}, \bigwedge_{i=1}^k \tilde C_i \right] + 
            \Prob\left[ \overline C_{k+1} \mid \bigwedge_{i=1}^k \tilde C_i \right] 
            \Prob\left[ A \mid \overline C_{k+1}, \bigwedge_{i=1}^k \tilde C_i \right] \\
            &= &&\Prob\left[ A \mid \bigwedge_{i=1}^{k} \tilde C_i \right]
        \end{alignedat}
    \end{equation*}
    and the induction step is proven.
\end{proof} 

\begin{lem}
    \label{lemma:upper-bound-high-girth}
    Let $H=(V,E)$ be an $r$-uniform hypergraph of girth at least $4$ with minimum degree $\delta \geq \ln r + 2$ and maximum degree $\Delta$.
    Let $\eta = \Delta/(\delta - 1)$ and let $t$ be a positive integer satisfying
    \begin{equation}
        \label{eq:01-triangle-free-upper-bound}
        \frac{t!}{t+1} \geq \euler r^2 \eta^{t+1} (\ln r + 1 + \euler^{-1})^{t+1}~.
    \end{equation}
    Then $H$ has a $t$-shallow hitting edge set.
\end{lem}
\begin{proof}    
    We build an edge set $M$ with the following two-step random experiment.
    In the first step, we pick each edge in $E$ independently at random with probability $(\ln r + 1)/(\delta - 1)$.
    In the second step, we pick for every vertex $v$ that is not covered after the first step an incident edge independently uniformly at random.
    Clearly, $M$ is a hitting edge set.
    For an edge set $F$, we denote by $V(F)$ the set of vertices $\bigcup_{e \in F} e$.
    For a set $\mathcal{F}$ of edge sets, we denote by $V(\mathcal{F})$ the set of vertices $\bigcup_{F \in \mathcal{F}} V(F)$.

    For each edge $e \in E$, we define $X_e$ to be the event that $e$ is picked in the first step.
    For \emph{each} vertex $v$ (independently of the results of the first step), we pick an incident edge independently uniformly at random and denote by $Y_v \in \incident(v)$ the random variable that indicates which edge is picked, and by $Y_{v,e}$ the event that $Y_v = e$.
    Then, 
    \[
        W_e = \bigvee_{v \in e} \left( Y_{v,e} \wedge \bigwedge_{e' \in \incident(v)} \overline X_{e'} \right)
    \]
    denotes the event that $e$ is picked in the second step and $Z_e = X_e \vee W_e$ is the event that $e$ is picked in the first or second step.
    Define the set $\mathcal{F}$ to be the set of all edge sets $F \subseteq E$ of size $t+1$ such that there exists a vertex $v \in V$ with $F \subseteq \incident(v)$, i.e.,
    \begin{equation*}
        \mathcal{F} = \{F \subseteq E \colon |F|=t+1, F \subseteq \incident(v) \text{ for some }  v \in V\}~.
    \end{equation*}
    For a set $F \in \mathcal{F}$, we denote the event that $F \subseteq M$ by $A_F$, i.e., $A_F = \bigwedge_{e \in F} Z_e$.
    These shall be the bad events in the Lovász Local Lemma below.
    
    We need the following observation, which easily follows from $H$ having girth at least $4$.
    \begin{obs}
        \label{obs:girth-four}
        For any two distinct edges $e,e' \in E$,
        \begin{enumerate}
            \item $|e \cap e'| \leq 1$ and
            \item if $e \cap e' = \{u\}$ for some $u \in V$ then there exists no edge $f \in E \setminus \{e,e'\}$ with $e \cap f = \{v\}$ and $e' \cap f =\{v'\}$ for $v,v' \in E \setminus \{u\}$.
        \end{enumerate}
    \end{obs}
    
    Now, we need to bound the probability that a bad event $A_F$ occurs.
    Therefore, we show that $\Prob\left[Z_e \mid \bigwedge_{e' \in F'} Z_{e'}\right] \leq \Prob[Z_e]$, which will be done in the following two claims and using Observation~\ref{obs:girth-four}.

    \begin{claim}
    \label{claim:probabilities-bad-events-subclaim}
        Let $F' \neq \emptyset$ be an edge set with $\bigcap_{e' \in F'} e' \neq \emptyset$.
        Then, $\Prob[\bigwedge_{e' \in F'} Z_{e'} \mid \bigwedge_{e' \in F'} \overline X_{e'}] \leq \Prob[\bigwedge_{e' \in F'} Z_{e'}]$.
    \end{claim}
    \begin{claimproof}
        Let $v$ be a vertex in $\bigcap_{e' \in F'} e'$, which is unique for $|F'| \geq 2$ (Observation~\ref{obs:girth-four}).
        First, we bound the right-hand side:
        \[
            \Prob\left[\bigwedge_{e' \in F'} Z_{e'}\right]
            \geq \Prob\left[\bigwedge_{e' \in F'} X_{e'}\right]
            = \prod_{e' \in F'} \Prob[X_{e'}]
            = \left(\frac{\ln r + 1}{\delta - 1}\right)^{|F'|}~,
        \]
        where we used that the events $X_{e'}$ are independent.
        Now, we bound the left-hand side as follows:
        \begin{equation}
        \label{eq:probabilities-bad-events-subclaim}
            \begin{alignedat}{2}
                & &&\Prob\left[\bigwedge_{e' \in F'} Z_{e'} \mid \bigwedge_{e' \in F'} \overline X_{e'}\right]
                = \Prob\left[\bigwedge_{e' \in F'} W_{e'} \mid \bigwedge_{e' \in F'} \overline X_{e'}\right]\\
                &\leq && \Prob\left[Y_v \in F' \wedge \bigwedge_{g \in \incident(v) \setminus F'} \overline X_g\right] \Prob\left[\bigwedge_{e' \in F'} W_{e'} \mid Y_v \in F', \bigwedge_{g \in \incident(v)} \overline X_g\right]\\
                & && + \Prob\left[\bigwedge_{e' \in F'} W_{e'} \mid \left(Y_v \notin F' \vee \bigvee_{g \in \incident(v)\setminus F'} X_g \right), \bigwedge_{e' \in F'} \overline X_{e'}\right]
            \end{alignedat}
        \end{equation}
        Then, we explicitly calculate the probabilities using Observation~\ref{obs:girth-four}.
        First, it holds that $\Prob\left[Y_v \in F' \wedge \bigwedge_{g \in \incident(v) \setminus F'} \overline X_g\right] \leq \Prob\left[Y_v \in F'\right] \leq |F'|/\delta$.
        
        For the calculation of the other probabilities, we need the following bounds.
        Let $e,e' \in \incident(v)$ be two distinct edges and $v' \in e' \setminus\{v\}$.
        Then,
        \begin{equation*}
            \begin{split}
                &\Prob\left[Y_{v',e'} \wedge \bigwedge_{e'' \in \incident(v')} \overline X_{e''} \mid \overline X_{e'}\right]
                = \Prob\left[Y_{v',e'} \wedge \bigwedge_{e'' \in \incident(v') \setminus\{e'\}} \overline X_{e''}\right]\\
                &\leq \frac{1}{\delta} \left(1-\frac{\ln r - 1}{\delta - 1}\right)^{\delta - 1}
                \leq \frac{1}{\euler \delta r}
            \end{split}
        \end{equation*}
        and
        \[
            \Prob\left[W_{e'} \mid Y_{v} = e, \overline X_{e'}\right]
            \leq \sum_{v' \in e' \setminus \{v\}} \Prob\left[Y_{v',e'} \wedge \bigwedge_{e'' \in \incident{v'}} X_{e''} \mid \overline X_{e'}\right]
            \leq \frac{r-1}{\euler \delta r} \leq \frac{1}{\euler \delta}~.
        \]
        Here we used $\left(1-(\ln r + 1)/(\delta - 1)\right)^{\delta - 1} \leq 1/(\euler r)$ for $\delta \geq \ln r + 2$.

        Finally, we can bound the other probabilities occurring in Equation~\ref{eq:probabilities-bad-events-subclaim}.
        Let $e \in F'$ be fixed. Then,
        \begin{equation*}
            \begin{split}
                &\Prob\left[ \bigwedge_{e' \in F'} W_{e'}\mid Y_v \in F', \bigwedge_{g \in \incident(v)} \overline X_g \right]
                \leq \Prob\left[\bigwedge_{e' \in F' \setminus\{e\}} W_{e'} \mid Y_v = e, \bigwedge_{g \in \incident(v)} \overline X_g\right]\\
                &= \prod_{e' \in F' \setminus\{e\}} \Prob\left[W_{e'} \mid Y_v = e, \overline X_{e'}\right]
                \leq \left(\frac{1}{\euler \delta}\right)^{|F'|-1}~.
            \end{split}
        \end{equation*}
        Analogously, we bound using some fixed $e \in \incident(v) \setminus F'$:
        \begin{equation*}
            \begin{split}
                &\Prob\left[\bigwedge_{e' \in F'} W_{e'} \mid \left(Y_v \notin F' \vee \bigvee_{g \in \incident(v)\setminus F'} X_g \right), \bigwedge_{e' \in F'} \overline X_{e'}\right]\\
                &\leq \prod_{e' \in F'} \Prob\left[W_{e'} \mid \left(Y_{v} = e \vee \bigvee_{g \in \incident(v)\setminus F'} X_g\right), \overline X_{e'}\right]
                \leq \left(\frac{1}{\euler\delta}\right)^{|F'|}~.
            \end{split}
        \end{equation*}

        In total, we get
        \begin{equation*}
                \Prob\left[\bigwedge_{e' \in F'} Z_{e'} \mid \bigwedge_{e' \in F'} \overline X_{e'}\right]
                \leq (\euler |F'| + 1) \left(\frac{1}{\euler \delta}\right)^{|F'|}
                \leq (\euler |F'| + 1) \left(\frac{1}{\euler (\delta - 1)}\right)^{|F'|}~.
        \end{equation*}

        Hence, in order to show the claim, we only need to show that $(\euler (\ln r + 1))^{|F'|} \geq \euler |F'| + 1$.
        Since $r \geq 2$, it suffices to show that $(\euler (\ln 2 + 1))^{|F'|} \geq \euler |F'| + 1$, which holds for all $|F'| \geq 1$.
        This finishes the proof of this claim.
    \end{claimproof}
    
    \begin{claim}
        \label{claim:probabilities-bad-events}
        For every edge set $F'$ and every edge $e \notin F'$ with $e \cap (\bigcap_{e' \in F'} e') \neq \emptyset$ it holds that $\Prob\left[Z_e \mid \bigwedge_{e' \in F'} Z_{e'}\right] \leq \Prob[Z_e]$.
    \end{claim}
    \begin{claimproof}
        For $F'=\emptyset$, the claim obviously holds.
        Thus, let $F' \neq \emptyset$.
        Let $v$ be the vertex in $e \cap (\bigcap_{e' \in F'} e')$.
        By Observation~\ref{obs:girth-four}, the vertex $v$ is unique.
        We show that the conditions of Lemma~\ref{lemma:helper-lemma-triangle-free-probabilities} hold with $A := Z_e$, $B := \bigwedge_{e' \in F'} Z_{e'}$, $C_1 := \bigvee_{e' \in F'} X_{e'}$, $C_2 := \overline X_e$ and $C_3 := \overline Y_{v,e}$.

        First, we show that Condition~\ref{itm:helper-lemma-eq01} of Lemma~\ref{lemma:helper-lemma-triangle-free-probabilities} holds.
        Let $\tilde C_2 \in \{C_2, \overline C_2\}$ and $\tilde C_3 \in \{C_3, \overline C_3\}$ be arbitrary.
        First, note that $\Prob[A \mid B, C_1, \tilde C_2, \tilde C_3] = \Prob[A \mid C_1, \tilde C_2, \tilde C_3]$ as if one edge in $F'$ is picked in the first step, then $A$ is independent of $B$.
        Secondly, $\Prob[A \mid B, \overline C_1, \overline C_2, \tilde C_3] = \Prob[A \mid \overline C_1, \overline C_2, \tilde C_3]$ as if $e$ is picked in the first step then $A$ is independent of $B$.
        Thirdly, $\Prob[A \mid B, \overline C_1, C_2, C_3] = \Prob[A \mid \overline C_1, C_2, C_3]$ as if $Y_v \neq e$ then $A$ is independent of $B$.
        And finally, $\Prob[A \mid B, \overline C_1, C_2, \overline C_3] = \Prob[A \mid \overline C_1, C_2, \overline C_3]$ as if none of the edges in $F' \cup \{e\}$ is picked in the first step and $Y_v = e$, then $A$ is independent of $B$.
        In each case we used that $H$ has girth at least $4$ and Observation~\ref{obs:girth-four}.
        Therefore, Condition~\ref{itm:helper-lemma-eq01} of Lemma~\ref{lemma:helper-lemma-triangle-free-probabilities} is satisfied.

        Now, we show that Condition~\ref{itm:helper-lemma-eq02} of Lemma~\ref{lemma:helper-lemma-triangle-free-probabilities} holds.
        First, note that $\Prob[A \mid \tilde C_1, \tilde C_2, C_3] \leq \Prob[A \mid \tilde C_1, \tilde C_2]$ and $\Prob[A \mid \tilde C_1, C_2] \leq \Prob[A \mid \tilde C_1]$ for all $\tilde C_1 \in \{C_1, \overline C_1\}$ and $\tilde C_2 \in \{C_2, \overline C_2\}$ hold since $C_2$ and $C_3$ only have negative occurrences in $A$.
        Moreover, $\Prob[A \mid C_1] \leq \Prob[A]$ since each $X_{e'}$ for $e' \in F'$ only has negative occurrences in $A$.
        This satisfies Condition~\ref{itm:helper-lemma-eq02} of Lemma~\ref{lemma:helper-lemma-triangle-free-probabilities}.
        
        Now, we show that Condition~\ref{itm:helper-lemma-eq03} of Lemma~\ref{lemma:helper-lemma-triangle-free-probabilities} holds.
        Note that $C_2$ and $C_3$ only have positive occurrences in $B$ and therefore, $\Prob[B \mid \tilde C_1, \tilde C_2, C_3] \geq \Prob[B \mid \tilde C_1, \tilde C_2]$ and $\Prob[B \mid \tilde C_1, C_2] \geq \Prob[B \mid \tilde C_1]$ for all $\tilde C_1 \in \{C_1, \overline C_1\}$ and $\tilde C_2 \in \{C_2, \overline C_2\}$.
        Hence, we only have to show that $\Prob[B \mid C_1] \geq \Prob[B]$, which is equivalent to $\Prob[B \mid \overline C_1] \leq \Prob[B]$ and is done in Claim~\ref{claim:probabilities-bad-events-subclaim}.
        
        Therefore, all Conditions of Lemma~\ref{lemma:helper-lemma-triangle-free-probabilities} are satisfied and we can conclude that $\Prob[A \mid B] \leq \Prob[A]$, i.e., $\Prob\left[Z_e \mid \bigwedge_{e' \in F'} Z_{e'}\right] \leq \Prob[Z_e]$, which finishes the proof of the claim.
    \end{claimproof}
    
    Using Claim~\ref{claim:probabilities-bad-events}, we can now bound the probability of the bad events $A_F$.
    The probability that an edge $e$ is picked in the first or second step is bounded by
    \[ 
        \Prob[Z_e] = \Prob[X_e] + \Prob[W_e]
        \leq \frac{\ln r + 1}{\delta - 1} + \left( 1 - \frac{\ln r + 1}{\delta - 1} \right)^\delta \cdot \frac{r}{\delta}
        \leq \frac{\ln r + 1 + \euler^{-1}}{\delta-1}~.
    \]
    Here, we used our assumption $\delta \geq \ln r + 2$.
    By Claim~\ref{claim:probabilities-bad-events}, we can bound the probability of a bad event $A_F$ by
    \[
        \Prob[A_F] = \Prob\left[ \bigwedge_{e \in F} Z_e \right]
        \leq \prod_{e \in F} \Prob[Z_e]
        \leq \left( \frac{\ln r + 1 + \euler^{-1}}{\delta - 1} \right)^{t+1} =: p~.
    \]

    Now, we have to bound the degree of dependence of the events $A_F$ with $F \in \mathcal{F}$.
    We use the same dependency graph as in the proof of Lemma~\ref{lemma:upper-bound-uniform-regular-hypergraph}, i.e., two bad events $A_F$ and $A_{F'}$ are adjacent in the dependency graph if and only if $V(F) \cap V(F') \neq \emptyset$.
    To show that this dependency graph satisfies the conditions of Remark~\ref{rem:weak-dependency-graph}, we use the following claim.
    
    \begin{claim}
        \label{claim:dependence-bad-events}
        For every set $\mathcal{F}'$ of edge sets in $\Fcal$ and an edge set $F \in \Fcal$ with $V(F) \cap V(\mathcal{F}') = \emptyset$ it holds that $\Prob\left[A_F \mid \bigwedge_{F' \in \mathcal{F}'} \overline A_{F'}\right] \leq \Prob[A_F]$.
    \end{claim}
    \begin{claimproof}
        Let $S \subseteq E$ be the set of edges $g \in E$ with $g \cap V(F) \neq \emptyset$ and $g \cap V(\mathcal{F}') \neq \emptyset$.
        We show that the conditions of Lemma~\ref{lemma:helper-lemma-triangle-free-probabilities} hold with $A := A_F$, $B := \bigwedge_{F' \in \mathcal{F}'} \overline A_{F'}$ and $\{C_1,\dots,C_n\} := \{X_g\}_{g \in S}$.
        Then, by Lemma~\ref{lemma:helper-lemma-triangle-free-probabilities}, this claim follows.

        First of all, Condition~\ref{itm:helper-lemma-eq01} holds since by conditioning on $\bigwedge_{i=1}^n \tilde C_i$ for $\tilde C_i \in \{C_i, \overline C_i\}$, $A$ is independent of $B$.
        Secondly, Condition~\ref{itm:helper-lemma-eq02} and Condition~\ref{itm:helper-lemma-eq03} hold since $C_j = X_g$, for some $g \in S$, only has negative occurrences in $A = A_F$ and only positive occurrences in $B = \bigwedge_{F' \in \mathcal{F}'} \overline A_{F'}$.
    \end{claimproof}

    It therefore follows that the graph constructed above satisfies the conditions of Remark~\ref{rem:weak-dependency-graph} and is a valid dependency graph.
    Analogous to the proof of Lemma~\ref{lemma:upper-bound-uniform-regular-hypergraph}, we can bound the degree of dependence $d$ by
    \[
        d + 1 \leq (t+1)\Delta r^2 \binom{\Delta-1}{t} \leq \frac{t+1}{t!} r^2 \Delta^{t+1}~.
    \]
    To apply the Lovász Local Lemma, we calculate $\euler p (d + 1)$.
    With Equation~\eqref{eq:01-triangle-free-upper-bound} and $\eta = \Delta / (\delta-1)$ it follows that
    \begin{equation*}
        \euler p (d + 1) 
        \leq \euler \cdot \left( \frac{\ln r + 1 + \euler^{-1}}{\delta - 1} \right)^{t+1} \cdot \frac{t+1}{t!} r^2 \Delta^{t+1}
        \leq \euler r^2 \eta^{t+1} \left( \ln r + 1 + \euler^{-1} \right)^{t+1} \frac{t+1}{t!}
        \leq 1~.
    \end{equation*}
    By the Lovász Local Lemma, the probability that no event $A_F$ with $F \in \mathcal{F}$ occurs is greater than zero.
    Thus, there exists an edge set $M \subseteq E$ such that every vertex has degree $1 \leq \degree_M(v) \leq t$ in $M$.
\end{proof}

Using Lemma~\ref{lemma:upper-bound-high-girth}, we find an asymptotic upper bound for $t(\mu, r)$, i.e., the largest integer such that every $r$-uniform $\mu$-near-regular hypergraph of girth at least $4$ has a $t$-shallow hitting edge set.
Let $W(x)$ be the Lambert $W$-function, i.e., the unique function defined by $x = W(x) \exp(W(x))$ for all real numbers $x \geq -e^{-1}$.
\begin{thm}
    \label{thm:upper-bound-high-girth}
    Let $H$ be an $r$-uniform $\mu$-near-regular hypergraph of girth at least $4$.
    Then, $H$ has a $t(\mu, r)$-shallow hitting edge set, where 
    \[
        t(\mu, r) = c \mu \ln(r) (1+o_{\mu r}(1)) = O(\mu \ln r)
    \]
    and $c = \euler^{1+W(2/(\euler \mu))} = \euler + o_\mu(1)$.
\end{thm}

\begin{proof}
    Let $\delta$ be the minimum degree and $\Delta$ be the maximum degree of $H=(V,E)$.
    If $\delta \leq \ln r + 2$ then $\Delta \leq \mu (\ln r + 2)$ and the edge set $E$ is a $t$-shallow hitting edge set.
    Thereby, assume that $\delta > \ln r + 2$. 
    Let $\eta = \Delta / (\delta-1)$.
    Note that $\mu \leq \eta = \mu \delta / (\delta - 1) = (1+o_r(1)) \mu$ for $\delta > \ln r + 2$.
    We show that
    \begin{equation}
        \begin{split}
            \label{eq:t(r)-high-girth}
            t &= 1 + \left\lceil \euler^{1+h(\eta, r)} \eta (\ln r + 1 + \euler^{-1}) \right\rceil
            = \euler^{1+W(2/(\euler \eta))} \eta \ln(r) (1+o_{\eta r}(1))
        \end{split}
    \end{equation}
    with
    \[
        h(\eta, r)
        = W\left(\frac{2}{\euler} \cdot \left( \frac{1}{\eta} + \frac{\ln(\eta (\ln r + 1 + \euler^{-1}))}{\eta (\ln r + 1 + \euler^{-1})} \right)\right)
        = W\left( \frac{2}{\euler \eta} \right) (1+o_{\eta r}(1))
    \]
    satisfies Equation~\eqref{eq:01-triangle-free-upper-bound}.
    Indeed, with Stirling's Formula $t! \geq \sqrt{2 \pi t} (t/\euler)^t$ and with $(t+1)/\sqrt{2 \pi t} \leq t$ for $t \geq 1$ it follows that
    \begin{multline*}
            \euler \frac{t+1}{t!} r^2 \eta^{t+1} (\ln r + 1 + \euler^{-1})^{t+1}
            \leq \euler \frac{t+1}{\sqrt{2 \pi t}} r^2 \eta (\ln r + 1 + \euler^{-1}) \left( \frac{\eta(\ln r + 1 + \euler^{-1})}{t/\euler} \right)^t\\
            \leq \euler^2 r^2 \eta^2 (\ln r + 1 + \euler^{-1})^2 \left( \frac{\eta(\ln r + 1 + \euler^{-1})}{(t-1)/\euler} \right)^{t-1}~.
    \end{multline*}
    We define $s = \euler \eta (\ln r + 1 + \euler^{-1})$ and $a = (rs)^2$.
    Observe that $h(\eta, r)=W((\ln a)/s)$.
    By using Equation~\eqref{eq:t(r)-high-girth}, we get
    \begin{equation*}
        \euler \frac{t+1}{t!} r^2 \eta^{t+1} (\ln r + 1 + \euler^{-1})^{t+1}
        \leq a \exp\left(-s \cdot h(\eta, r) \euler^{h(\eta, r)}\right)
        = a \exp\left(-s \cdot \frac{\ln a}{s}\right)
        = 1~.
    \end{equation*}

    Therefore, we have $t=\euler^{1+W(2/(\euler \eta))} \eta \ln(r) (1+o_{\eta r}(1))$.
    With $\mu \leq \eta \leq (1+o_r(1))\mu$, the theorem follows.
\end{proof}

For $\mu=1$ we get the following corollary.
\begin{cor}
    \label{cor:upper-bound-high-girth-regular}
    Every $r$-uniform regular hypergraph of girth at least $4$ has a $t(r)$-shallow hitting edge set, where
    \[
        t(r) = c \ln(r) (1+o_{r}(1)) = O(\ln r)
    \]
    and $c = \euler^{1+W(2/\euler)} = 4.319\dots$.
\end{cor}

\subsection{Lower bounds for uniform regular partite hypergraphs}
\label{subsec:lower-bound-shallow-hitting}

For the lower bounds, we seek to find for infinitely many values of $r$, an integer $t = t(r)$ (ideally large in terms of $r$) and an $r$-uniform $r$-partite regular hypergraph $H_t$ with no $(t-1)$-shallow hitting edge set.
In fact, we shall construct $H_t$ to be $t$-regular, which translates the absence of a $(t-1)$-shallow hitting edge set in $H_t$ to the absence of a proper vertex $2$-coloring of its dual $H_t^*$.
By using the dualization, we can employ constructions of $t$-uniform non-$2$-colorable hypergraphs with maximum degree $r$ (ideally small in terms of $t$), which are a classical topic initiated by Erd\H{o}s and Lov\'asz~\cite{EL73}.
The currently best bound is due to Kostochka and R\"odl, where we cite here only the case of two colors and girth at least four.

\begin{thm}[Kostochka and R\"odl~\cite{KR10}]\label{thm:3-chromatic-hypergraph}
    For any $t \geq 2$ there exists a $t$-uniform $r$-regular non-$2$-colorable hypergraph $H$ of girth at least four, with $r = \lceil t2^{t-1} \ln 2 \rceil$.
    Moreover, $H$ is the union of $r$ perfect matchings.
\end{thm}

\begin{rem}
    Kostochka and R\"odl do not claim that $H$ is the union of $r$ perfect matchings, but their construction is easily tweaked to admit this extra property.
    In fact, they start with any $T$-uniform $r$-regular $H_0$ of large enough girth (girth at least four is enough for us) for some specific large value of $T$ of the form $T = m_0 \cdot t$.
    Then, each hyperedge of $H_0$ is replaced by a certain (randomly chosen) matching of $m_0$ hyperedges of size $t$ each.
    Now we can choose $H_0$ to be the union of $r$ perfect matchings (a random choice on a large enough vertex set will do), which implies that also $H$ is the union of $r$ perfect matchings.
\end{rem}

\begin{cor}
\label{cor:lower-bound-girth-four}
    For any $t \geq 2$ there exists an $r$-uniform $r$-partite $t$-regular hypergraph of girth at least four that has no $(t-1)$-shallow hitting edge set, where $r = \lceil t2^{t-1} \ln 2 \rceil$.
\end{cor}
\begin{proof}
    Let $H_t$ be the dual of the hypergraph $H$ from Theorem~\ref{thm:3-chromatic-hypergraph}.
    Then $H_t$ is $r$-partite (hence $r$-uniform) since $H$ is the union of $r$ perfect matchings, $H_t$ is $t$-regular since $H$ is $t$-uniform, and $H_t$ has girth at least four since $H$ has girth at least four.
    Finally, a $(t-1)$-shallow hitting edge set $M$ in $H_t$ would correspond to a vertex set $A$ of $H$ such that every edge of $H$ contains at least one vertex of $A$ (since $M$ is hitting) and at most $t-1$ vertices of $A$ (since $A$ is $(t-1)$-shallow).
    Thus $(A, V_H-A)$ would be a proper vertex $2$-coloring of $H$, which is a contradiction.
\end{proof}

In the next theorem, we construct an $r$-uniform $r$-partite regular hypergraph which only has $t$-shallow hitting edge sets for $t \geq \log_2(r+1)$.

\begin{figure}
    \centering
    \includegraphics{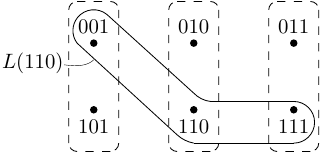}
    \caption{The construction of the hypergraph $H'$ of Theorem~\ref{theorem:01-truncated-projective-space} with $t=2$. The vertices are elements of $\mathbb{F}_2^{t+1}$. Two vertices $x$ and $x'$ are in the same part if $x+x'=x_0$ with $x_0=(1,0,\dots,0)^T$. The set $L(a)$ is the set of all $x \in \mathbb{F}_2^{t+1}\setminus\{0\}$ with $a^T x = 0$. The set of edges of $H'$ is the set of all $L(a)$ that do not contain $x_0$.}
    \label{fig:truncated-projective-space}
\end{figure}
\begin{thm} 
    \label{theorem:01-truncated-projective-space}
    Let $t \geq 2$ be an integer.
    There exists an $r$-uniform $r$-partite $2^{t-1}$-regular hypergraph with two vertices per part that has no $(t-1)$-shallow hitting edge set, where $r=2^t-1$, i.e., $t = \log_2(r+1)$.
\end{thm}
\begin{proof}
    This construction is based on the construction of projective spaces of order $2$.
    
    Let $\mathbb{F}_2$ be the finite field of two elements and let $V = \mathbb{F}_2^{t+1} \setminus \{0\}$ be the set of vertices.
    For a vector $a \in \mathbb{F}_2^{t+1}$, we define $L(a)$ to be the set of vertices $x \in V$ with $a^T x = 0$.
    Then, we define the edge set to be $E=\{L(a) \mid a \in \mathbb{F}_2^{t+1} \setminus \{0\}\}$ and define the hypergraph by $H=(V,E)$.
    The vertices of $H$ correspond to the $1$-dimensional subspaces of $\mathbb{F}_2^{t+1}$ and the edges correspond to the $t$-dimensional subspaces of $\mathbb{F}_2^{t+1}$.
    
    Let $L(a)$ be an edge.
    Since $a \neq 0$, there are exactly $2^t$ solutions to the equation $a^T x = 0$ for $x \in \mathbb{F}_2^{t+1}$ and thus, there are $2^t - 1$ solutions to $a^T x = 0$ with $x \in \mathbb{F}_2^{t+1} \setminus \{0\}$.
    Hence, each edge has size $2^t - 1$ and $H$ is $r$-uniform with $r=2^t - 1$.

    Let $x \in V$ be a vertex.
    Analogous, there are $2^t - 1$ solutions of $a^T x = 0$ with $a \in \mathbb{F}_2^{t+1} \setminus \{0\}$ and thus, $H$ is $(2^t - 1)$-regular.

    Observe that each $t$ edges have a common vertex.
    Indeed, let $L(a_1),\dots,L(a_t)$ be $t$ edges of $H$.
    Consider the matrix $A := (a_1,a_2,\dots,a_t) \in \mathbb{F}_2^{(t+1) \times t}$.
    Then, $A^T$ has rank at most $t$ and the nullspace of $A^T$ has dimension at least $1$.
    Thus, there is at least one nonzero solution $x \in V$ to $A^T x = 0$.
    Then, $x$ clearly satisfies all equations $a_i^T x = 0$ for $i=1,\dots,t$ and thus, $x$ is incident to all edges $L(a_1),\dots,L(a_t)$.

    From $H$, we construct an $r$-uniform $r$-partite $2^{t-1}$-regular hypergraph $H'$ by ``truncating'' the hypergraph $H$.
    Let $x_0=(1,0,\dots,0)^T \in V$.
    To obtain $H'=(V',E')$ from $H=(V,E)$, we remove the vertex $x_0 \in V$ and all incident edges of $x_0$.
    See also Figure~\ref{fig:truncated-projective-space} for the construction of $H'$ with $t=2$.
    Note that $H'$ is $r$-uniform and every set of $t$ edges has a vertex in common.
    Moreover, $H'$ has $|V'|=2^{t+1}-2=2r$ vertices.
    We need to show that $H'$ is $2^{t-1}$-regular and $r$-partite with parts of size $2$.
    Since every $t$ edges have a vertex in common, it follows that there exists no $(t-1)$-shallow hitting edge set.
    
    We define that two vertices $x,x' \in V$ are in the same part of $H'$ if $x + x' = x_0$.
    Clearly, each edge $L(a)$ of $H$ which is incident to both $x$ and $x'$ is incident to $x_0$, since $a^T x_0 = a^T x + a^T x' = 0 + 0 = 0$.
    Thus, no edge of $H'$ contains two vertices of the same part and $H'$ is $r$-uniform $r$-partite with parts of size $2$.

    Moreover, for every vertex $x \in V \setminus \{x_0\}$ in $H$ there are $2^{t-1}-1$ edges $L(a)$ which are incident to both $x_0$ and $x$, which can be seen as follows.
    Since $x_0=(1,0,\dots,0)$ is a solution of $a^T x_0 = 0$, it must hold that $a_1 = 0$.
    Then, there are $2^{t-1}-1$ solutions $a$ to $a^T x = 0$ with $a_1=0$, since $x \in V \setminus \{x_0\}$.
    Thus, we removed $2^{t-1}-1$ incident edges for each vertex $x \in V \setminus \{x_0\}$ from $H$ to obtain $H'$.
    Then, we can conclude that $H'$ is $d$-regular with $d=2^t-1-(2^{t-1}-1) = 2^{t-1}$.
\end{proof}

\section{Maximum Shallow Edge Sets}
\label{sec:maximum-shallow-edge-sets}

In this section, we consider the maximum possible size of shallow edge sets in $r$-uniform $\mu$-near-regular hypergraphs.
In comparison to $t$-shallow hitting edge sets, we relax the condition $1 \leq \degree(v) \leq t$ to $\degree(v) \leq t$ for all vertices $v \in V$.
Additionally, we are not interested in minimizing $t$ for a given $r$ and $\mu$ (as for shallow hitting edge sets) but in maximizing the size of a $t$-shallow edge set for given $t$, $r$ and $\mu$ in $r$-uniform $\mu$-near-regular hypergraphs.
We show that every $r$-uniform $\mu$-near-regular hypergraph $H$ with $n$ vertices has a $t$-shallow edge set of size
\begin{equation*}
    \Omega\left( \frac{n t}{\mu r^{1+1/t}}\right)~,
\end{equation*}
whenever $\Delta \geq t$ for the maximum degree $\Delta$ of $H$.
Moreover, we show in Section~\ref{subsec:04-03-combinatorial-designs}, that the lower bound above is tight for regular $r$-uniform hypergraphs.

For this, we use the following framework.
For a hypergraph $H$, let $\nu_t(H)$ be the size of a maximum $t$-shallow edge set in $H$.
Note that $\nu_1(H)$ is the size of a maximum matching in $H$.
Let $H$ be an $r$-uniform hypergraph with $n$ vertices.
Observe that $\nu_t(H) \leq n t / r$ is an upper bound on the size of a $t$-shallow edge set in $H$ since every vertex can be covered at most $t$ times and every edge has size $r$.
If $H$ is an $r$-uniform hypergraph with $n$ vertices, we define by $\eta_t(H) = \nu_t(H) / (n t / r)$.
We are particularly interested in bounding $\eta_t(H)$ for such hypergraphs $H$.
For this, let $\mathcal{H}_t(r,\mu)$ be the set of all $r$-uniform $\mu$-near-regular hypergraphs with $\Delta(H) \geq t$.
We only consider such hypergraphs with $\Delta(H) \geq t$ since for $\Delta(H) < t$, the whole edge set is trivially a maximum $t$-shallow edge set.
We define
\begin{equation*}
    \eta_t(r,\mu) = \inf_{H \in \mathcal{H}_t(r,\mu)} \eta_t(H)
\end{equation*}
Moreover, we define $\eta_t(r) = \eta_t(r, 1)$.
Observe that $0 \leq \eta_t(H) \leq 1$ for every hypergraph $H \in \mathcal{H}_t(r,\mu)$ and thus $0 \leq \eta_t(r,\mu) \leq 1$ for every positive integer $r$ and every real number $\mu \geq 1$.

We show in Section~\ref{subsec:03-01-lower-bound} that
\begin{equation*}
    \eta_t(r, \mu) \geq \frac{1-o_t(1)}{\euler} \cdot \mu^{-1} r^{-1/t}~.
\end{equation*}
This result is tight for $\mu = 1$ up to a constant factor.
In fact, we show that for every real number $\epsilon > 0$ there exist infinitely many positive integers $r$ such that
\begin{equation*}
    \eta_t(r) \leq (1+\epsilon) \cdot r^{-1/t}
\end{equation*}
by providing a construction through combinatorial designs.

\subsection{A lower bound for near-regular hypergraphs}
\label{subsec:03-01-lower-bound}

\begin{lem}
\label{lemma:03-01-lower-bound}
Let $r \geq 2$, $t$ and $k$ be positive integers.
Let $H=(V,E)$ be an $r$-uniform hypergraph of maximum degree $\Delta$.
If
\begin{equation}
\label{eq:03-01-lower-bound}
    \frac{t!}{t+1} \left( \frac{k}{\Delta} \right)^t \geq \euler r
\end{equation}
then there exists a partition of $E$ into $k$ disjoint $t$-shallow edge sets $M_1,\dots, M_k$.
\end{lem}

\begin{proof}
We consider the following random experiment.
For each edge $e \in E$, we pick a color $\chi(e) \in \{1, \dots, k\}$ uniformly at random.
The colors correspond to the sets $M_i$, $i=1, \dots, k$.
We use the Lovász Local Lemma to prove that there exists an edge coloring such that no vertex has $t+1$ incident edges of the same color.
For this, we define $\mathcal{F}$ to be the set of all edge sets $F \subseteq E$ of size $t+1$ such that there exists a vertex $v \in V$ with $F \subseteq \incident(v)$, i.e.
\begin{equation*}
    \mathcal{F} = \{F \subseteq E \colon |F|=t+1, F \subseteq \incident(v) \text{ for some }  v \in V\}~.
\end{equation*}
For a set $F \in \mathcal{F}$, denote by $A_F$ the event that all edges in $F$ received the same color.
Clearly,
\begin{equation*}
    \Prob[A_F] = \frac{1}{k^t} =: p~.
\end{equation*}
We construct a dependency graph $G_\text{D}=(V_\text{D}, E_\text{D})$ with $V_\text{D}=\{A_F \mid F \in \mathcal{F}\}$ and two events $A_F$ and $A_{F'}$ adjacent if and only if $F \cap F' \neq \emptyset$.
Then, each event $A_F$ is mutually independent of all non-adjacent events $A_{F'}$.

In the next step, we have to bound the degree of dependence $d$ of the events $A_F$ for $F \in \mathcal{F}$.
For an edge $e \in E$ and a vertex $v \in V$, denote by $\mathcal{F}_{e,v}$ the set of all edge sets $F \in \mathcal{F}$ with $e \in F$ and $F \subseteq \incident(v)$.
Note that the edge $e$ is contained in all sets $F \in \mathcal{F}_{e,v}$.
Since $F \subseteq \incident(v)$ for all $F \in \mathcal{F}_{e,v}$, there are at most $\binom{\Delta - 1}{t}$ ways to choose the remaining edges.
Hence, the size of each $\mathcal{F}_{e,v}$ can be bounded by
\begin{equation*}
    |\mathcal{F}_{e,v}| \leq \binom{\Delta - 1}{t}~.
\end{equation*}
Let $A_F$ be an arbitrary but fixed event with $F \in \mathcal{F}$.
For every adjacent event $A_{F'}$, it holds that $F \cap F' \neq \emptyset$.
Thus, there must be an edge $e$ in $F \cap F'$.
Moreover, by the definition of $\mathcal{F}$, there exists a vertex $v$ such that $F \subseteq \incident(v)$.
Then, the degree of dependence $d$ can be bounded by 
\begin{equation*}
    d+1 
    \leq \max_{F \in \mathcal{F}} \left| \bigcup\limits_{e \in F} \bigcup\limits_{v \in e} \mathcal{F}_{e,v} \right|
    \leq (t+1)\cdot r \cdot \binom{\Delta - 1}{t} \leq \frac{t+1}{t!} r \Delta^t~.
\end{equation*}
To apply the Lovász Local Lemma, we calculate $\euler p (d+1)$ and use Equation~\ref{eq:03-01-lower-bound}:
\begin{equation*}
    \euler p (d+1) \leq \frac{\euler}{k^t} \cdot \frac{t+1}{t!} r \Delta^t \leq 1~.
\end{equation*}
By the Lovász Local Lemma, the probability that no event $A_F$ with $F \in \mathcal{F}$ occurs is greater than zero.
Thus, there exists an edge coloring $\chi \colon E \to \{1, \dots, k\}$ such that no vertex has $t+1$ incident edges of the same color.
For $i=1, \dots, k$, we define $M_i$ to be the set of all edges of color $i$.
Then, each set $M_i$ is $t$-shallow and it holds that $M_1 \dot\cup\cdots\dot\cup M_k = E$.
\end{proof}

\begin{lem}
\label{lemma:03-01-lower-bound-2}
Let $r \geq 2$ and $t$ be positive integers.
Let $H=(V,E)$ be an $r$-uniform hypergraph of maximum degree $\Delta \geq t$.
Then, there exists a partition of $E$ into $k$ disjoint $t$-shallow edge sets $M_1,\dots, M_k$ for
\begin{equation*}
    k = \frac{\euler \Delta r^{1/t}}{t} \cdot (1+o_t(1))~.
\end{equation*}
\end{lem}

\begin{proof}
We prove the lemma by showing that
\begin{equation*}
    k = \left\lceil \frac{\euler \Delta}{t} \left( \frac{t+1}{\sqrt{2 \pi t}} \euler r \right)^{1/t}\right\rceil
\end{equation*}
satisfies Equation~\ref{eq:03-01-lower-bound}.
For that, we use Stirling's Formula $t! \geq \sqrt{2 \pi t} (t/\euler)^t$ and the definition of $k$:
\begin{equation*}
    \left( \frac{k}{\Delta}\right)^t \frac{t!}{t+1}
    \geq \left( \frac{t k}{\Delta \euler}\right)^t \frac{\sqrt{2 \pi t}}{t+1}
    \geq \frac{t+1}{\sqrt{2 \pi t}} \euler r \cdot \frac{\sqrt{2 \pi t}}{t+1}
    \geq \euler r~.
\end{equation*}
By Lemma~\ref{lemma:03-01-lower-bound}, there exist $k$ disjoint $t$-shallow edge sets $M_1, \dots, M_k$ with $M_1 \dot\cup\cdots\dot\cup M_k = E$.
With 
\begin{equation*}
    \left( \frac{t+1}{\sqrt{2 \pi t}} \euler \right)^{1/t} = 1 + o_t(1)
\end{equation*}
it holds that
\begin{equation*}
    k = \frac{\euler \Delta r^{1/t}}{t} \cdot (1+o_t(1))~.
\end{equation*}
\end{proof}

\begin{thm}
\label{theorem:03-01-lower-bound}
Let $n$, $t$ and $r\geq 2$ be positive integers and $\mu \geq 1$ be a real number.
Let $H=(V,E)$ be an $r$-uniform $\mu$-near-regular hypergraph with $n$ vertices and maximum degree $\Delta \geq t$.
Then there exists a $t$-shallow edge set of size at least
\begin{equation*}
    \frac{n t}{\euler \mu r^{1+1/t}} \cdot (1-o_t(1))~.
\end{equation*}
\end{thm}

\begin{proof}
Let $\Delta$ be the maximum degree and $\delta$ be the minimum degree of $H$.
By Lemma~\ref{lemma:03-01-lower-bound-2}, for 
\begin{equation*}
    k = \frac{\euler \Delta r^{1/t}}{t}\cdot (1+o_t(1))
\end{equation*}
there exist $k$ disjoint $t$-shallow edge sets $M_1, \dots, M_k$ with $M_1 \dot\cup\cdots\dot\cup M_k = E$.
It holds that
\begin{equation*}
    |E|=|M_1| + \dots + |M_k| \geq n \delta / r~.
\end{equation*}
By the pigeonhole principle, there exists an edge set $M_i$ of size
\begin{equation*}
    \max_{i \in [k]} |M_i| 
    \geq \frac{n \delta / r}{k} 
    = \frac{n \delta t}{\euler \Delta r^{1+1/t}} \frac{1}{1+o_t(1)}
    = \frac{n t}{\euler \mu r^{1+1/t}} \cdot (1-o_t(1))~.
\end{equation*}
\end{proof}

With Theorem~\ref{theorem:03-01-lower-bound}, we obtain the following lower bound for $\eta_t(r,\mu)$.
\begin{cor}
Let $t$ and $r \geq 2$ be positive integers and let $\mu \geq 1$ be a real number.
Then,
\begin{equation*}
    \eta_t(r,\mu) \geq \frac{1-o_t(1)}{\euler} \cdot \mu^{-1} r^{-1/t}~.
\end{equation*}
\end{cor}

\subsection{Upper bound via combinatorial designs}
\label{subsec:04-03-combinatorial-designs}

In this section, we show that there exist $r$-uniform ($r$-partite) regular hypergraphs that only have $t$-shallow edge sets of small size.
This result follows from the existence of combinatorial designs, proved by Keevash~\cite{keevash2018existence}.

\begin{dfn}
Let $t, v, k$ and $\lambda$ be positive integers.
A set of \emph{points} $V$ with a multiset $\mathcal{B}$ of subsets of $V$ (called \emph{blocks}) is called a \emph{$t$-$(v,k,\lambda)$-design} if
\begin{enumerate}
    \item $|V|=v$ and
    \item $|B|=k$ for each block $B \in \mathcal{B}$ and
    \item for each set $U \subseteq V$ of size $t$ there exist exactly $\lambda$ blocks $B \in \mathcal{B}$ with $U \subseteq B$.
\end{enumerate}
\end{dfn}

First, we state a well-known property of combinatorial designs.
This theorem gives a necessary condition for the existence of combinatorial designs.
\begin{thm}[\cite{handbookdesigns}]
\label{theorem:04-03-necessary-condition-design}
Let $(V,\mathcal{B})$ be a $t$-$(v,k,\lambda)$-design.
If $I \subseteq V$ is a set of size $0 \leq |I| = i \leq t$, then the number of blocks containing $I$ is
\begin{equation*}
    r_i = \lambda \binom{v-i}{t-i} \left/ \binom{k-i}{t-i}\right.~.
\end{equation*}
\end{thm}

By Theorem~\ref{theorem:04-03-necessary-condition-design}, we obtain a necessary condition for the existence of a $t$-$(v,k,\lambda)$-design: for all integers $i$ with $0 \leq i \leq t-1$ it must hold that $\binom{k-i}{t-i}$ divides $\lambda \binom{v-i}{t-i}$.
Keevash~\cite{keevash2018existence} recently showed that this condition is also sufficient for large enough $v$.
Here, a $t$-$(v,k,\lambda)$-design $(V, \mathcal{B})$ is called \emph{resolvable} if the set $\mathcal{B}$ of blocks can be partitioned into perfect matchings.

\begin{thm}[Keevash~\cite{keevash2018existence}]
\label{theorem:04-03-keevash}
Suppose $k \geq t \geq 1$ and $\lambda$ are fixed and $v > v_0(k,t,\lambda)$ is sufficiently large such that $k \mid v$ and $\binom{k-i}{t-i} \mid \lambda \binom{v-i}{t-i}$ for all integers $i$ with $0 \leq i \leq t-1$.
Then there exists a resolvable $t$-$(v,k,\lambda)$-design.
\end{thm}

It immediately follows from Theorem~\ref{theorem:04-03-keevash} that for any positive integers $t$ and $k$ with $k \geq t \geq 1$ there exist $t$-$(n k, k, 1)$-designs for infinitely many $n > n_0(k,t)$.
\begin{cor}
\label{corollary:04-03-keevash}
Suppose $k \geq t \geq 1$ are fixed.
Then there exist infinitely many positive integers $n$ such that there exists a resolvable $t$-$(n k, k, 1)$-design.
\end{cor}
\begin{proof}
We use Theorem~\ref{theorem:04-03-keevash} to prove this corollary.
Clearly, $k \mid n k$ and thus, the first condition is satisfied.
Therefore, it suffices to show that there exist infinitely many positive integers $n$ which satisfy $\binom{k-i}{t-i} \mid \binom{n k-i}{t-i}$ for all integers $i$ with $0 \leq i \leq t-1$.

We claim that for all positive integers $\mu$,
\begin{equation*}
    n= 1 + \mu \cdot k (k-1) (k-2) \cdots (k-t+1)
\end{equation*}
satisfies the divisibility conditions.
Indeed, for all integers $i$ with $0 \leq i \leq t-1$,
\begin{equation*}
    \frac{n k - i}{k - i} = n + \frac{n-1}{k-i} i
\end{equation*}
is a positive integer and therefore,
\begin{equation*}
    \frac{\binom{n k - i}{t - i}}{\binom{k - i}{t - i}}
    = \frac{(n k - i)(n k - i - 1) \cdots (n k - t+1)}{(k - i)(k - i -1)\cdots (k-t+1)}
\end{equation*}
is a positive integer for all integers $i=0,1,\dots,t-1$.
\end{proof}

\begin{thm}
\label{theorem:04-03-upper-bound}
Let $t$ and $k$ be positive integers with $k > t \geq 1$.
Then there exist infinitely many positive integers $n$ such that there exists an $r$-uniform $r$-partite $k$-regular hypergraph $H$ with parts of size $n$ and maximum $t$-shallow edge set of size $t$, where $r=\binom{n k - 1}{t}/\binom{k-1}{t}$.
This shows
\begin{equation*}
    \eta_t(r) \leq \eta_t(H) \leq \frac{1}{1-t/k} \cdot r^{-1/t}
\end{equation*}
for the specified values of $r$.
\end{thm}
\begin{proof}
Let $H=(V,\mathcal{B})$ be a resolvable $(t+1)$-$(n k,k,1)$-design.
This exists due to Corollary~\ref{corollary:04-03-keevash} for infinitely many positive integers $n$.
That is, $|V| = n k$, each block $B \in \mathcal{B}$ has size $|B| = k$ and each set of $t+1$ elements of $V$ is contained in exactly one block in $\mathcal{B}$.
Then, the hypergraph $H$ is $k$-uniform, $r$-regular (because of Theorem~\ref{theorem:04-03-necessary-condition-design}) and has $n k$ vertices and $n r$ edges.
Define $H^*=(V^*,E^*)$ to be the dual hypergraph of $H$.
Then, $H^*$ has $n r$ vertices and $n k$ edges.
Moreover, $H^*$ is $r$-uniform and $k$-regular.
Since the design $H=(V,\mathcal{B})$ is resolvable, the hypergraph $H^*$ is $r$-partite with parts of size $n$.
Note that every set of $t+1$ edges in $H^*$ has a vertex that is incident to all $t+1$ edges.
Thus, the maximum $t$-shallow edge set has size $t$, i.e. $\nu_t(H) = t$.
With
\begin{equation*}
    r = \frac{\binom{n k-1}{t}}{\binom{k-1}{t}}
    = \frac{(n k -1)(n k -2)\cdots (n k-t)}{(k-1)(k-2)\cdots(k-t)}
    \leq \left( \frac{n k}{k-t} \right)^t
\end{equation*}
we obtain $n \geq r^{1/t}(1-t/k)$ and therefore
\begin{equation*}
    \eta_t(r) \leq \eta_t(H) = \frac{t}{tn} = \frac{1}{n} 
    \leq \frac{1}{1-t/k} \cdot r^{-1/t}~.
\end{equation*}
\end{proof}
For a fixed positive integer $t$, $k$ can be chosen arbitrarily large.
Therefore, the following corollary follows immediate.
\begin{cor}
Let $t$ be a fixed positive integer.
For every real number $\epsilon > 0$ there exist infinitely many positive integers $r$ such that
\begin{equation*}
    \eta_t(r) \leq (1+\epsilon) r^{-1/t}~.
\end{equation*}
\end{cor}
\begin{proof}
We choose $k$ such that $t/k \leq \epsilon/(1+\epsilon)$ and apply Theorem~\ref{theorem:04-03-upper-bound}.
\end{proof}
Thus, this upper bound asymptotically matches the lower bound obtained in Section~\ref{subsec:03-01-lower-bound} up to a constant factor.

\section{Minimum Degree Conditions}
\label{sec:minimum-degree-conditions}

\subsection{Uniform hypergraphs}
\label{subsec:uniform-hypergraphs}

In this section, we study sufficient conditions on the minimum degree of uniform hypergraphs for the existence of $t$-shallow hitting edge sets.
Let $H=(V,E)$ be an $r$-uniform hypergraph.
If $\hat e$ is an $(r-1)$-set of vertices of $H$, the \emph{neighborhood} $N_{r-1}(\hat e)$ of $\hat e$ is the set of all vertices $v$ such that $\hat e \cup \{v\}$ is an edge in $H$.
We define the \emph{minimum degree} $\delta_{r-1}(H)$ of an $r$-uniform hypergraph to be the minimum of $|N_{r-1}(\hat e)|$ over all $(r-1)$-sets $\hat e$ of vertices in $H$.
This is also called the \emph{minimum co-degree} of $H$.

The problem of finding minimum degree thresholds for the existence of perfect matchings in $r$-uniform hypergraphs is a classical topic in extremal graph theory.
These questions are referred to as Dirac-type problems as Dirac~\cite{dirac1952some} showed that every graph on $n$ vertices with minimum degree at least $n/2$ contains a Hamiltonian cycle (and therefore a perfect matching if $n$ is even).
R{\"o}dl, Ruci{\'n}ski and Szemer{\'e}di~\cite{rodl2006dirac} started the study of Dirac-type problems for $3$-uniform hypergraphs.
K\"{u}hn and Osthus~\cite{kuhnmatchings} and R{\"o}dl, Ruci{\'n}ski and Szemer{\'e}di~\cite{rodl2006perfect} improved the minimum degree conditions for $r$-uniform hypergraphs until R{\"o}dl, Ruci{\'n}ski and Szemer{\'e}di~\cite{rodl2009perfect} settled the problem by proving a tight minimum degree condition for the existence of a perfect matching in $r$-uniform hypergraphs, stated in Theorem~\ref{thm:minimum-degree-condition-matching-uniform-hypergraph}.
See also~\cite{zhao2016recent} for a survey on Dirac-type problems in hypergraphs.
\begin{thm}[R{\"o}dl, Ruci{\'n}ski and Szemer{\'e}di~\cite{rodl2009perfect}]
\label{thm:minimum-degree-condition-matching-uniform-hypergraph}
    Let $r \geq 3$ be an integer and let $n$ be sufficiently large and divisible by $r$.
    If $H$ is an $r$-uniform hypergraph with $n$ vertices of minimum degree
    \[
        \delta_{r-1}(H) \geq \begin{cases}
          n/2 + 3 - r, & \text{if $r/2$ is even and $n/r$ is odd}\\
          n/2 + 5/2 - r, & \text{if $r$ is odd and $(n-1)/2$ is odd}\\
          n/2 + 3/2 - r, & \text{if $r$ is odd and $(n-1)/2$ is even}\\
          n/2 + 2 - r, & \text{otherwise}
        \end{cases}
    \]
    then $H$ has a perfect matching.
    Moreover, the minimum degree condition is tight.
\end{thm}

Recall that perfect matchings are exactly $1$-shallow hitting edge sets.
Generalizing this, we seek to find as tight as possible minimum degree conditions for the existence of $t$-shallow hitting edge sets in $r$-uniform hypergraphs for fixed $t$ and $r$.
Let us start on the negative side by constructing an $r$-uniform hypergraph $H$ with large minimum degree $\delta_{r-1}(H)$ that has no $t$-shallow hitting edge set.
\begin{thm}
\label{theorem:02-09-uniform-lower-bound}
For all positive integers $r \geq 2$, $t \geq 2$ and $n$ there exists an $r$-uniform hypergraph $H=(V,E)$ with $|V| = n$ vertices and
\begin{equation*}
    \delta_{r-1}(H) \geq \frac{n}{(r-1)t+1} - 1
\end{equation*}
that has no $t$-shallow hitting edge set.
\end{thm}
\begin{proof}
Let $k=(r-1)t+1$ and $\delta = n/k-1$.
We define an $r$-uniform hypergraph $H=(V,E)$ with $V=A \dot\cup B$ where
\begin{equation*}
    |A| = \lceil \delta \rceil \quad \text{and} \quad |B| = n - \lceil \delta \rceil~.
\end{equation*}
Note that $|V|=|A| + |B| = n$ and $|A| < n/k$ and $|B| > (k-1) n/k$.
We define that $e \subseteq V$ is an edge in $H$ if and only if $|e|=r$ and $e$ has a vertex in $A$, i.e., $e \cap A \neq \emptyset$.
Clearly, $H$ has minimum degree $\delta_{r-1}(H) \geq \delta$.
Assume, for contradiction, that there exists a $t$-shallow hitting edge set $M$ of $H$.
Since every vertex in $A$ is covered at most $t$ times, we have
\begin{equation*}
    |M| \leq t |A| < \frac{t n}{k}~.
\end{equation*}
On the other hand,
\begin{equation*}
    |M| \geq \frac{|B|}{r-1} > \frac{(k-1) n}{(r-1) k} = \frac{t n}{k}~,
\end{equation*}
since every vertex in $B$ is covered at least once and at most $r-1$ vertices of $B$ are covered by the same edge in $M$.
Thus, $tn/k < |M| < tn/k$ is the desired contradiction.
\end{proof}

On the positive side, if the minimum degree of $H$ is just one larger than the bound in Theorem~\ref{theorem:02-09-uniform-lower-bound}, then $t$-shallow hitting edge sets always exist.
In other words, the following theorem is best-possible.

\begin{thm}
\label{theorem:02-09-large-degree-uniform-hypergraph}
Let $t \geq 2$ be an integer.
If $H=(V,E)$ is an $r$-uniform hypergraph with $|V| = n$ vertices which satisfies
\begin{equation*}
    \delta_{r-1}(H) \geq \frac{n}{(r-1)t+1}~,
\end{equation*}
then $H$ has a $t$-shallow hitting edge set.
\end{thm}

\begin{proof}
    Let $k=(r-1)t+1$.
    First of all, observe that $n \geq r$ as $\delta_{r-1}(H) = 0$ for $1 \leq n < r$, contradicting the assumption $\delta_{r-1}(H) \geq n/k > 0$.
    Let $M \subseteq E$ be a $t$-shallow edge set of $H$.
    For $i=0,1,\dots, t$, let $n_i(M)$ be the number of vertices $v \in V$ with $\degree_M(v)=i$.
    We define $M \subseteq E$ to be a $t$-shallow edge set that minimizes $n_0(M)$ on condition that each edge $e \in M$ contains at most one vertex $v \in e$ with $\degree_M(v) \geq 2$, i.e., $|\{v \in e \mid \degree_M(v) \geq 2\}| \leq 1$ for all $e \in M$.
    Moreover, we set $n_i=n_i(M)$ for all $i=0,1,\dots,t$ and $M$ as defined.

    \begin{claim}
        \label{claim:upper-bound-nt}
        If $n_0 > 0$ then $n_t < \left\lceil n/k \right\rceil$.
    \end{claim}
    \begin{claimproof}
        Assume that $n_t \geq \lceil n/k \rceil$.
        Each vertex of degree $t$ in $M$ has $(r-1)t$ neighbors of degree $1$ in $M$, all of them distinct.
        Thus, $n_1 \geq (r-1) t n_t$ and 
        \begin{equation*}
            n \geq n_0 + n_1 + n_t > ((r-1)t+1) n_t \geq k \lceil n/k \rceil \geq n~,
        \end{equation*}
        a contradiction.
    \end{claimproof}

    \begin{claim}
        \label{claim:upper-bound_n0}
        $n_0 \leq r-2$.
    \end{claim}
    \begin{claimproof}
        Assume that $n_0 \geq r-1$.
        Let $\hat e$ be an $(r-1)$-set of vertices that is not covered by $M$.
        By Claim~\ref{claim:upper-bound-nt} it holds that $n_t < \lceil n/k \rceil$.
        Since $\delta_{r-1}(H) \geq \lceil n/k \rceil$, there exists a vertex $v \in N_{r-1}(\hat e)$ with $\degree_M(v) < t$.
        We define $e = \hat e \cup \{v\}$ and distinguish four cases.

        First, assume that $v$ is not covered by $M$.
        Then, $M'=M \cup \{e\}$ is a $t$-shallow edge set with $n_0(M') < n_0(M)$ and each edge $e' \in M'$ still has at most one vertex $v' \in e'$ of degree at least two, a contradiction.

        Secondly, assume that $\degree_M(v) = 1$ and \emph{the} edge $e' \in M$ incident to $v$ has no vertex of degree at least two in $M$.
        As in the first case, $M' = M \cup \{e\}$ leads to the desired contradiction.

        Thirdly, assume that $\degree_M(v) = 1$ and \emph{the} edge $e' \in M$ incident to $v$ has a vertex $v'$ of degree at least two in $M$.
        Then, define $M' = (M \cup \{e\}) - \{e'\}$.
        Then, $V(M') = (V(M) - V(e')) \cup V(e) \cup \{v'\}$ since $v'$ has degree at least one in $M'$.
        Thus, $|V(M')|=|V(M)|+1$ and $n_0(M') < n_0(M)$.
        As we deleted the edge $e'$, every edge in $M'$ still has at most one vertex of degree at least two in $M'$.
        This is the desired contradiction.

        It remains the case that $\degree_M(v) \geq 2$.
        We define $M' = M \cup \{e\}$.
        It clearly holds that $n_0(M') < n_0(M)$ and $M'$ is $t$-shallow.
        The edge $e$ has exactly one vertex of degree at least two in $M'$.
        Moreover, for each edge $e' \in M$ that is incident to $e$ it holds that $e \cap e' = \{v\}$.
        Therefore, each edge $e' \in M'$ has at most one vertex of degree at least two.
    \end{claimproof}

    We are now able to prove the theorem.
    If $n_0 = 0$, we are done.
    Hence, assume that $n_0 > 0$.
    By Claim~\ref{claim:upper-bound-nt} we have $n_t < \lceil n/k \rceil$ and by Claim~\ref{claim:upper-bound_n0} we have $n_0 \leq r-2$.
    We build an $(r-1)$-set $\hat e$ of vertices of degree less than $t$ that contains all vertices that are not covered by $M$.
    This set $\hat e$ exists since $n \geq r$ and therefore $n-n_t > (k-1)n/k = (r-1)tn/k \geq r-1$ for $t \geq 1$.
    Since $n_t < \lceil n/k \rceil$ and $\delta_{r-1}(H) \geq \lceil n/k \rceil$, there is a vertex $v \in N_{r-1}(\hat e)$ with $\degree_M(v) < t$.
    We define $e = \hat e \cup \{v\}$ and $M' = M \cup \{e\}$.
    Then, $M'$ is a $t$-shallow hitting edge set.
\end{proof}

\subsection{Uniform partite hypergraphs}
\label{subsec:uniform-partite-hypergraphs}

Now, let us consider $r$-uniform $r$-partite hypergraphs.
Recall that for $r=2$ these are exactly the bipartite graphs.
It is well-known that every bipartite graph with parts of size $n$ has a perfect matching if $\degree(v) \geq n/2$ for all vertices $v$.
The definition of the minimum degree of a bipartite graph is extendable to $r$-uniform $r$-partite hypergraphs in the following way, see for example~\cite{kuhnmatchings} and~\cite{aharoniperfectmatchings}.
Let $H=(V_1 \dot\cup \cdots \dot\cup V_r, E)$ be an $r$-uniform $r$-partite hypergraph.
We say that a set $\hat e \subseteq V$ of vertices of $H$ is \emph{legal} if $|\hat e \cap V_i| \leq 1$ for all parts $V_i$ of $H$.
If $\hat e \subseteq V$ is a set of vertices of $H$ with $|\hat e| = r-1$, the \emph{neighborhood} $N_{r-1}(\hat e)$ of $\hat e$ is the set of all vertices $v$ such that $\hat e \cup \{v\}$ is an edge in $H$.
Given an $r$-uniform $r$-partite hypergraph $H$, we define the \emph{minimum degree} $\delta'_{r-1}(H)$ to be the minimum of $|N_{r-1}(\hat e)|$ over all legal sets $\hat e$ of vertices in $H$ with $|\hat e| = r-1$.
The minimum degree $\delta'_{r-1}(H)$ is also called \emph{minimum co-degree}, but observe the important difference between $\delta'_{r-1}(H)$ for $r$-uniform $r$-partite hypergraphs $H$ considered here and $\delta_{r-1}(H)$ for $r$-uniform hypergraphs $H$ from Section~\ref{subsec:uniform-hypergraphs}.

Kühn and Osthus~\cite{kuhnmatchings} proved a sufficient bound on $\delta'_{r-1}(H)$ for the existence of perfect matchings in $r$-uniform $r$-partite hypergraphs.
This bound was improved in~\cite{aharoniperfectmatchings} as stated below.
\begin{thm}[Aharoni, Georgakopoulos and Spr\"{u}ssel~\cite{aharoniperfectmatchings}]
    \label{theorem:aharoni-matchings-upper-bound}
    Let $r \geq 2$ and let $H=(V = V_1 \dot\cup \cdots \dot\cup V_r, E)$ be an $r$-uniform $r$-partite hypergraph with parts of size $n$ such that for every legal $(r-1)$-set $\hat e$ contained in $V \setminus V_1$ we have $|N_{r-1}(\hat e)| > n/2$ and for every legal $(r-1)$-set $\hat f$ contained in $V \setminus V_r$ we have $|N_{r-1}(\hat f)| \geq n/2$.
    Then there exists a perfect matching in $H$.
\end{thm}
The upper bound in Theorem~\ref{theorem:aharoni-matchings-upper-bound} is tight up to the condition $|N_{r-1}(\hat e)| > n/2$ for every legal $(r-1)$-set $\hat e$ contained in $V \setminus V_1$.
It is not known whether this condition can be replaced by $|N_{r-1}(\hat e)| \geq n/2$ for every legal $(r-1)$-set $\hat e$.
On the other hand, there exist $r$-uniform $r$-partite hypergraphs $H$ with $\delta'_{r-1}(H) \geq n/2 - 1$ that have no perfect matching, due to a construction in~\cite{kuhnmatchings}, which shows that the minimum degree condition in Theorem~\ref{theorem:aharoni-matchings-upper-bound} is essentially tight.

\smallskip

In this section, we seek to extend these results to $t$-shallow hitting edge sets in $r$-uniform $r$-partite hypergraphs $H$ with parts of equal sizes and large minimum degree $\delta'_{r-1}(H)$.
For a start, we show a sufficient condition for the existence of $t$-shallow hitting edge sets in bipartite graphs (that is, $2$-uniform $2$-partite hypergraphs) of large minimum degree and show that our required lower bound on $\delta_1(H)$ is tight.
Afterwards, we extend this to $r$-uniform $r$-partite hypergraphs.

\begin{thm}[Berge and Las Vergnas~\cite{berge1978existence}]
    \label{thm:shallow-hitting-edge-sets-bipartite-graphs}
    Let $G=(A \dot\cup B, E)$ be a bipartite graph and $t$ be a positive integer.
    Then, $G$ has a $t$-shallow hitting edge set if and only if $|X| \leq t |N(X)|$ for all sets $X \subseteq A$ and $X \subseteq B$.
\end{thm}
\begin{thm}
\label{thm:02-00-minimum-vertex-degree}
Let $G=(A \dot\cup B, E)$ be a bipartite graph with $|A| \leq |B| \leq t |A|$.
Let $\delta_A$ respectively $\delta_B$ be the minimum degree of the vertices in $A$ respectively $B$.
If $t \delta_A + \delta_B \geq |A|$ and $\delta_A + t \delta_B \geq |B|$ then there exists a $t$-shallow hitting edge set.
\end{thm}
\begin{proof}
By Theorem~\ref{thm:shallow-hitting-edge-sets-bipartite-graphs}, we have to show that $|X| \leq t |N(X)|$ holds for all sets $X \subseteq A$ and all sets $X \subseteq B$.

We show that every set $X \subseteq A$ satisfies $|X| \leq t |N(X)|$.
First, assume that $1 \leq |X| \leq t \delta_A$.
Since every vertex in $X$ has at least $\delta_A$ neighbors, we have $|N(X)| \geq \delta_A$ and thus $|X| \leq t |N(X)|$.
Now assume that $|X| > t \delta_A$.
It follows that $|A \setminus X| < |A| - t \delta_A \leq \delta_B$.
Since every vertex in $B$ has at least $\delta_B$ neighbors in $A$, it must have a neighbor in $X$.
Thus, $N(X) = B$ and $t|N(X)| = t |B| \geq t |A| \geq |X|$.

Now, we show that every set $X \subseteq B$ satisfies $|X| \leq t |N(X)|$.
First, assume that $1 \leq |X| \leq t \delta_B$.
Since every vertex in $X$ has at least $\delta_B$ neighbors, we have $|N(X)| \geq \delta_B$ and thus $|X| \leq t |N(X)|$.
Now assume that $|X| > t \delta_B$.
It follows that $|B \setminus X| < |B| - t \delta_B \leq \delta_A$.
Since every vertex in $A$ has at least $\delta_A$ neighbors in $A$, it must have a neighbor in $X$.
Thus, $N(X) = A$ and $t|N(X)| = t |A| \geq |B| \geq |X|$.
\end{proof}

As a special case of Theorem~\ref{thm:02-00-minimum-vertex-degree}, assume that $|A|=|B|=n$.
It follows that if $G$ has minimum degree $\delta(G) \geq n/(t+1)$ then $G$ has a $t$-shallow hitting edge set.
Indeed,
\begin{equation*}
    t \delta_A + \delta_B \geq \frac{t n}{t+1} + \frac{n}{t+1} = n = |A|
\end{equation*}
and
\begin{equation*}
    \delta_A + t \delta_B \geq \frac{n}{t+1} + \frac{t n}{t+1} = n = |B|~.
\end{equation*}
To see that the condition $\delta(G) \geq n/(t+1)$ is tight, let $G'=(A \dot\cup B, E)$ be a bipartite graph with parts of size $|A|=|B|=n$ such that $t+1$ divides $n-1$.
Let $A = A_1 \dot\cup A_2$ and $B = B_1 \dot\cup B_2$ where $|A_1| = |B_2| = (n-1)/(t+1)$ and $|A_2| = |B_1| = (n t + 1)/(t+1)$ and note that that $|A|=|B|=n$.
The edges of $G'$ are exactly the pairs $\{a_1,b_1\}$ and $\{a_2, b_2\}$ with $a_1 \in A_1, a_2 \in A_2, b_1 \in B_1, b_2 \in B_2$.
The graph $G'$ has minimum degree $\delta(G') = (n-1)/(t+1) \geq  n/(t+1) - 1$ but does not contain a $t$-shallow hitting edge set as
\begin{equation*}
    t|N(B_1)| = t |A_1| = t \cdot \frac{n-1}{t+1} = \frac{t n + 1}{t + 1} - 1 = |B_1| - 1 < |B_1|~.
\end{equation*}

Now, let us turn to $t$-shallow hitting edge sets in $r$-uniform $r$-partite hypergraphs for $r > 2$.
By extending a construction from~\cite{kuhnmatchings}, we obtain the following construction of a hypergraph $H$ with large minimum degree $\delta'_{r-1}(H)$ but without a $t$-shallow hitting edge set.
\begin{thm}
\label{theorem:02-09-lower-bound}
For all positive integers $r \geq 2$, $t \geq 2$ and $n$ there exists an $r$-uniform $r$-partite hypergraph $H$ with parts of size $n$ and
\begin{equation*}
    \delta'_{r-1}(H) \geq \frac{n}{(r-1)t+1} - 1
\end{equation*}
that has no $t$-shallow hitting edge set.
\end{thm}
\begin{proof}
Let $\delta'=n/((r-1)t+1)-1$.
We define the $r$-partite $r$-uniform hypergraph $H=(V_1\dot\cup V_2 \dot\cup \cdots \dot\cup V_r, E)$ with $|V_i|=n$ for all $i=1,2,\dots,r$.
Let $V'_i$ be a subset of $V_i$ of size $|V'_i| = \lceil \delta' \rceil$, for $i=1,2,\dots,r$.
Note that $|V'_1 \cup V'_2 \cup\cdots\cup V'_r| < r (\delta'+1)$.
We define that $e=\{v_1,v_2,\dots,v_r\}$ is an edge in $H$ if and only if $v_1 \in V_1, v_2 \in V_2, \dots, v_r \in V_r$ and there is at least one $v_i$ with $v_i \in V'_i$.
Clearly, $H$ has minimum degree $\delta'_{r-1}(H) \geq \delta'$.
Assume, for contradiction, that there exists a $t$-shallow hitting edge set $M$ of $H$.
Then,
\begin{equation*}
    |M| \leq t|V'_1 \cup V'_2 \cup \cdots \cup V'_r|
    < t r (\delta' + 1)
\end{equation*}
since every vertex in $V'_1 \cup V'_2 \cup \cdots \cup V'_r$ is covered at most $t$ times.
On the other hand,
\begin{equation*}
\begin{split}
    |M| &\geq \frac{1}{r-1} \left|(V_1 \setminus V'_1) \cup (V_2 \setminus V'_2) \cup \cdots \cup (V_r \setminus V'_r)\right|
    > \frac{r n - r (\delta' + 1)}{r-1} \\
    &= \frac{r n t}{(r - 1) t + 1} = t r (\delta' + 1)
\end{split}
\end{equation*}
since every vertex in $V_i \setminus V'_i$ is covered at least once and at most $r-1$ vertices of $(V_1 \setminus V'_1) \cup \cdots \cup (V_r \setminus V'_r)$ are covered by the same edge in $M$.
Thus, $t r (\delta' + 1) < |M| < t r (\delta' + 1)$ is the desired contradiction.
\end{proof}

To obtain a condition on $\delta'_{r-1}(H)$ that is sufficient for the existence of $t$-shallow hitting edge sets, we first state a theorem from~\cite{kuhnmatchings}.
It says that if we relax the condition $\delta'_{r-1}(H) > n/2$ to $\delta'_{r-1}(H) \geq n/r$, then there exists an almost perfect matching, i.e., a matching that covers all but at most $r-2$ vertices from each part of $H$.
In~\cite{kuhnmatchings} it is shown that this result is tight as the minimum degree condition $\delta'_{r-1}(H) \geq n/r$ cannot be reduced. 
\begin{thm}[Kühn and Osthus~\cite{kuhnmatchings}]
\label{theorem:02-09-kuhn-matchings}
Let $H$ be an $r$-uniform $r$-partite hypergraph with parts of size $n$ and $\delta'_{r-1}(H) \geq n/r$.
Then $H$ has a matching that covers all but at most $r-2$ vertices in each part of $H$.
\end{thm}

As mentioned before, we seek to generalize Theorem~\ref{theorem:aharoni-matchings-upper-bound} to $t$-shallow hitting edge sets (cf.\ Theorem~\ref{theorem:02-09-upper-bound} below).
To this end, we first show in Lemma~\ref{lemma:02-09-lemma-3} below a generalization of Theorem~\ref{theorem:02-09-kuhn-matchings} to $t$-shallow edge sets that are not necessarily hitting but cover almost all vertices of each part.
To show Lemma~\ref{lemma:02-09-lemma-3}, we consider any $r$-uniform $r$-partite hypergraph $H$ with large minimum degree $\delta'_{r-1}(H)$ and a $t$-shallow edge set $M$ satisfying certain properties, leaving as few vertices as possible uncovered.
We then deduce properties of such a hypergraph $H$ (cf.\ Lemma~\ref{lemma:02-09-lemma-2} below) and show that these lead to a contradiction in Lemma~\ref{lemma:02-09-lemma-3}.
Using Lemma~\ref{lemma:02-09-lemma-3} and raising the minimum degree condition by $1$, we can show that there exists a $t$-shallow hitting edge set in hypergraphs with large minimum degree $\delta'_{r-1}(H)$.
In the first step, we show some properties of $t$-shallow hitting edge sets in $r$-uniform $r$-partite hypergraphs.

\begin{lem}
    \label{lemma:02-09-lemma-1}
    Let $r \geq 2$ and $t \geq 2$ be positive integers and $H=(V_1 \dot\cup \cdots \dot\cup V_r, E)$ be an $r$-uniform $r$-partite hypergraph with parts of size $n$.
    Let $M$ be a $t$-shallow edge set that satisfies the following properties.
    Here, let $n_{s,i}$ be the number of vertices $v \in V_i$ with $\degree_M(v)=s$.
    \begin{enumerate}[label = (\arabic*)]
        \item\label{itm:02-09-lemma-1-property-M-1} There exists a non-negative integer $n_0$ with $n_{0,1} = n_{0,2} = \cdots = n_{0,r} = n_0$.
        \item\label{itm:02-09-lemma-1-property-M-2} $|M| \leq \left\lceil n/k \right\rceil (t-1) + n - n_0$.
    \end{enumerate}
    Then, with $k=(r-1)t+1$, all of the following holds.
    \begin{enumerate}[label = (\Alph*)]
        \item\label{itm:02-09-lemma-1-res-1}
        $n_{t,i} \leq \lceil n/k \rceil$ for all $i=1,2,\dots,r$.
        
        \item\label{itm:02-09-lemma-1-res-2}
        If there exists a vertex $v \in V_i$ with $2 \leq \degree_M(v) < t$, then $n_{t,i} \leq \lceil n/k \rceil - 1$.
        
        \item\label{itm:02-09-lemma-1-res-3}
        If $|M| < \lceil n/k \rceil (t-1) + n - n_0$, then $n_{t,i} \leq \lceil n/k \rceil - 1$ for all $i=1,2,\dots,r$.
        
        \item\label{itm:02-09-lemma-1-res-4}
        If $|M| = \lceil n/k \rceil (t-1) + n - n_0$ and $|\{v \in e \mid \degree_M(v) \geq 2\}| \leq 1$ for all $e \in M$, then $n_{t,i}=\lceil n/k \rceil$ for all $i=1,2,\dots,r$.
    \end{enumerate}
\end{lem}

\begin{proof}
    \begin{enumerate}[label = (\Alph*)]
        \item
        For all $i=1,2,\dots,r$, there exist $n-n_0-n_{t,i}$ vertices in $V_i$ that are covered at least once by $M$ but less than $t$ times.
        Thus, $|M| \geq t n_{t,i} + (n - n_0 - n_{t,i}) = (t-1) n_{t,i} + n - n_0$.
        By comparing this inequality with Property~\ref{itm:02-09-lemma-1-property-M-2}, we get $n_{t,i} \leq \lceil n/k \rceil$.
        
        \item
        Using the same argument, we have $|M| \geq t n_{t,i} + (n - n_0 - n_{t,i}) + 1 = (t-1) n_{t,i} + n - n_0 + 1$.
        By comparing this inequality with Property~\ref{itm:02-09-lemma-1-property-M-2}, we get $n_{t,i} \leq \lceil n/k \rceil - 1/(t-1) < \lceil n/k \rceil$ and thus $n_{t,i} \leq \lceil n/k \rceil - 1$.
        
        \item
        Using the same argument, we have $|M| \geq (t-1) n_{t,i} + n - n_0$.
        Since $|M| < \lceil n/k \rceil (t-1) + n - n_0$ it holds that $n_{t,i} < \lceil n/k \rceil$ and thus $n_{t,i} \leq \lceil n/k \rceil - 1$.
        
        \item 
        Let $n_s := n_{s,i}$ for all $s=1,2,\dots,t$ and some arbitrary $i \in \{1,2,\dots,r\}$, and let $n_{\geq 2}=n_2 + n_3 + \cdots + n_t$.
        Our first goal is to upper-bound $n_1$ in the general case $n_t \leq \lceil n/k \rceil$ and in the case $n_t \leq \lceil n/k \rceil - 1$.
        For the general case, we upper-bound the size of $M$ by $|M| \leq n_1 + t n_{\geq 2} = (t-1) n_{\geq 2} + n - n_0$ and with $|M|=\lceil n/k \rceil (t-1) + n - n_0$ it follows that $n_{\geq 2} \geq \lceil n/k \rceil$.
        Thus, we have $n_1 = n - n_0 - n_{\geq 2} \leq n - n_0 - \lceil n/k \rceil$.
        
        In the next step, we show that if $n_t \leq \lceil n/k \rceil - 1$ then $n_{\geq 2} \geq \lceil n/k \rceil + 1$.
        First note that $t \geq 3$ since for $t = 2$ it holds that $\lceil n/k \rceil + n - n_0 = |M|=n_1 + 2 n_2 = n - n_0 + n_2 < \lceil n/k \rceil + n - n_0$, a contradiction.
        For $t \geq 3$, we have
        \begin{equation*}
        \begin{split}
            \left\lceil \frac{n}{k} \right\rceil (t-1) + n - n_0 
            &= |M| = n - n_0 + n_2 + 2 n_3 + \cdots + (t-1) n_t\\
            &\leq n - n_0 + (t-2) (n_2 + n_3 + \cdots + n_{t-1}) + (t-1) n_t
        \end{split}
        \end{equation*}
        and thus
        \begin{equation*}
            (t-1) \left( \left\lceil \frac{n}{k} \right\rceil - n_t \right) \leq (t-2) (n_2 + n_3 + \cdots + n_{t-1})~.
        \end{equation*}
        It follows that 
        \begin{equation*}
            n_{\geq 2} \geq n_t + \frac{t-1}{t-2} \left( \left\lceil \frac{n}{k} \right\rceil - n_t \right)
            = \frac{1}{t-2} \left( (t-1) \left\lceil \frac{n}{k} \right\rceil - n_t \right)
            > \left\lceil \frac{n}{k} \right\rceil
        \end{equation*}
        and therefore $n_1 \leq n - n_0 - \lceil n/k \rceil - 1$ if $n_t \leq \lceil n/k \rceil - 1$.
        
        Now, we prove the claim.
        Assume, for contradiction, that there exists a part $V_j$ with $n_{t,j} \leq \lceil n/k \rceil - 1$.
        Then, $n_{1,i} \leq n - n_0 - \lceil n/k \rceil$ for all $i=1,2,\dots,r$ and $n_{1,j} = n - n_0 - (n_{2,j}+n_{3,j}+\cdots+n_{t,j}) \leq n - n_0 - \lceil n/k \rceil - 1$ since $n_{t,j} \leq \lceil n/k \rceil - 1$ and thus $n_{2,j}+n_{3,j}+\cdots+n_{t,j} \geq \lceil n/k \rceil + 1$.
        In the next step, observe that
        \begin{equation*}
            (r-1) |M| \leq \sum_{i=1}^r n_{1,i}~,
        \end{equation*}
        since each edge in $M$ covers at least $(r-1)$ vertices that are covered once by $M$.
        This follows since each edge $e \in M$ has at most one vertex $v \in e$ with $\degree_M(v) \geq 2$.
        We bound the $n_{1,i}$'s in the sum and use $|M| = \lceil n/k \rceil (t-1) + n - n_0$ and obtain
        \begin{equation*}
            (r-1)\left(\left\lceil \frac{n}{k} \right\rceil (t-1) + n - n_0 \right) \leq \sum_{i=1}^r n_{1,i} \leq r \left(n - n_0 - \left\lceil \frac{n}{k} \right\rceil \right) - 1~.
        \end{equation*}
        Simplifying this inequality and using $k=(r-1)t+1$ and we get
        \begin{equation*}
            \left\lceil \frac{n}{k} \right\rceil k + 1 \leq n - n_0~,
        \end{equation*}
        a contradiction since $\lceil n/k \rceil \cdot k \geq n$ and $n_0 \geq 0$.\qedhere
    \end{enumerate}
\end{proof}

As the next step, we show some properties of $t$-shallow edge sets with specific other properties that maximize the number of covered vertices.
The intuition for the next lemma is, if the hypergraph $H$ has large minimum degree but no $t$-shallow edge set (with specific properties) that covers almost all vertices of each part, then $H$ is not far from the hypergraph constructed in Theorem~\ref{theorem:02-09-lower-bound}.
In the following, given a hypergraph $H=(V,E)$ and a subset of edges $M \subseteq E$, we define $V(M) \subseteq V$ to be the union of all edges in $M$.
\begin{lem}
\label{lemma:02-09-lemma-2}
Let $r \geq 2$ and $t \geq 2$ be positive integers and $H=(V_1 \dot\cup \cdots \dot\cup V_r, E)$ be an $r$-uniform $r$-partite hypergraph with parts of size $n$ and minimum degree $\delta'_{r-1}(H) \geq \lceil n/k \rceil$ where $k=(r-1)t+1$.
Let $M$ be a $t$-shallow edge set that satisfies the following properties and maximizes the number of covered vertices in $H$ with respect to these properties.
Here, let $n_{s,i}$ be the number of vertices $v \in V_i$ with $\degree_M(v)=s$.
\begin{enumerate}[label = (\arabic*)]
    \item\label{itm:02-09-lemma-2-property-M-1} There exists a non-negative integer $n_0$ such that $n_{0,1} = n_{0,2} = \cdots = n_{0,r} = n_0$.
    \item\label{itm:02-09-lemma-2-property-M-2} $|M| \leq \left\lceil n/k \right\rceil (t-1) + n - n_0$.
\end{enumerate}
If $n_0 \geq r-1$ then each of the following holds.
\begin{enumerate}[label = (\Alph*), start = 5]
    \item\label{itm:02-09-lemma-2-res-1}
    $N_{r-1}(\hat e) \subseteq V(M)$ for each legal $(r-1)$-set $\hat e$ of uncovered vertices of $H$.
    \item\label{itm:02-09-lemma-2-res-3}
    $|M|=\lceil n/k \rceil (t-1) + n - n_0$.
    \item\label{itm:02-09-lemma-2-res-2}
    For all edges $e \in M$ it holds that the number of vertices $v \in e$ that are covered at least twice by $M$ is at most one, i.e., $|\{v \in e \mid \degree_M(v) \geq 2\}| \leq 1$ for all edges $e \in M$.
\end{enumerate}
\end{lem}
\begin{proof}
We define $U_i = V_i \setminus V(M)$ to be the set of all uncovered vertices in $V_i$.
Moreover, we define $U = U_1\cup U_2 \cup \cdots \cup U_r$.
\begin{enumerate}[label = (\Alph*), start = 5]
    \item
    Let $\hat e$ be a legal $(r-1)$-set of uncovered vertices, that is $\hat e \subseteq U$.
    Assume that there exists a vertex $v \in U \cap N_{r-1}(\hat e)$.
    Let $e=\hat e \cup \{v\}$.
    Then, the edge set $M'=M \cup \{e\}$ is $t$-shallow.
    Let $U'_i=V_i \setminus V(M')$ be the set of vertices of $V_i$ that are not covered by $M'$.
    The edge set $M'$ satisfies Property~\ref{itm:02-09-lemma-1-property-M-1} with $n'_0 = |U'_1|=\cdots=|U'_r| = n_0 - 1$.
    Moreover, $M'$ satisfies Property~\ref{itm:02-09-lemma-1-property-M-2} since $|M'|=|M|+1=|M|+n_0-n'_0 \leq \lceil n/k \rceil (t-1) + n - n'_0$.
    Since $M'$ covers more vertices than $M$ and satisfies all properties, we have the desired contradiction.
    
    \item
    In the next step, we show that $|M|=\lceil n/k \rceil (t-1) + n - n_0$.
    Assume, for contradiction, that $|M| < \lceil n/k \rceil (t-1) + n - n_0$.
    By Lemma~\ref{lemma:02-09-lemma-1} Result~\ref{itm:02-09-lemma-1-res-3}, we have $n_{t,i} \leq \lceil n/k \rceil - 1$ for all parts $V_i$.
    Since $\delta'_{r-1}(H) \geq \lceil n/k \rceil$ and by Result~\ref{itm:02-09-lemma-2-res-1} of this lemma, each legal $(r-1)$-set $\hat e$ of vertices has a neighboring vertex that is covered at least once but less than $t$ times.
    For $i=1,2,\dots,r$, we build an $(r-1)$-set $\hat e_i$ of vertices of $U \setminus U_i$, such that no two sets $\hat e_i$ contain the same vertex.
    This is possible since $n_0 = |U_i| \geq r - 1$ for all $i = 1,2,\dots,r$.
    Let $v_i \in N_{r-1}(\hat e_i)$ be a vertex with $1 \leq \degree_M(v_i) < t$.
    Observe that $v_i \in V_i$ for all $i=1,2,\dots,r$.
    Define $e_i = \hat e_i \cup \{v_i\}$.
    We claim that
    \begin{equation*}
        M' = M \cup \{e_i \mid i=1,2,\dots,r\}
    \end{equation*}
    satisfies Property~\ref{itm:02-09-lemma-2-property-M-1} and Property~\ref{itm:02-09-lemma-2-property-M-2}.
    Clearly, $M'$ is $t$-shallow.
    Denote by $U'_i = V_i \setminus V(M')$ the set of vertices in $V_i$ that are not covered by $M'$.
    Then, $|U'_i|=|U_i|-(r-1) = n_0-r+1$ since the edges in $\{e_j \mid j \neq i, j=1,2,\dots,r\}$ cover $(r-1)$ vertices of $U_i$, but $e_i \cap U_i = \emptyset$.
    To show Property~\ref{itm:02-09-lemma-2-property-M-2}, note that $|M'|=|M|+r$ and $n'_0 := |U'_1| = \cdots = |U'_r| = n_0-r+1$.
    Thus, $|M'|=|M| + n_0 - n'_0 + 1 \leq \lceil n/k \rceil (t-1) + n - n'_0$.
    Since $M'$ covers more vertices than $M$ and satisfies all properties, we have the desired contradiction.
    
    \item
    We show that there exists no edge $e \in M$ such that at least two vertices of $e$ are covered at least twice by $M$.
    Assume that there exists an edge $e = \{u_1,u_2,\dots,u_r\} \in M$, with $u_i \in V_i$ for $i=1,2,\dots,r$, such that at least two vertices of $e$ are covered at least twice by $M$.
    Denote by $I \subseteq \{1,2,\dots,r\}$ the set of indices such that $\{u_i \mid i \in I\}$ is the set of vertices of $e$ that are covered at least twice by $M$.
    Then, $|I| \geq 2$.
    For each $i \in I$, we build a legal $(r-1)$-set $\hat e_i$ of vertices of $U \setminus U_i$, such that no two sets $\hat e_i$ contain the same vertex.
    This is possible since $n_0 = |U_i| \geq r - 1$ for all $i = 1,2,\dots,r$ and $|I| \leq r$.
    By Result~\ref{itm:02-09-lemma-1-res-1} of this lemma, $N_{r-1}(\hat e_i) \subseteq V_i \setminus U_i$.
    Denote by $W_i \subseteq V_i$ the set of vertices in $V_i$ that are covered exactly $t$ times by $M$.
    By Result~\ref{itm:02-09-lemma-1-res-1} of Lemma~\ref{lemma:02-09-lemma-1}, we have $|W_i| \leq \lceil n/k \rceil$ if $u_i \in W_i$ and, by Result~\ref{itm:02-09-lemma-1-res-2} of Lemma~\ref{lemma:02-09-lemma-1}, $|W_i| \leq \lceil n/k \rceil - 1$ if $u_i \notin W_i$, for $i \in I$.
    Since $|N_{r-1}(\hat e_i)| \geq \lceil n/k \rceil$, there exists a vertex $v_i \in N_{r-1}(\hat e_i)$ with $v_i \in (V_i \setminus (U_i \cup W_i)) \cup \{u_i\}$.
    That is, $1 \leq \degree_M(v_i) < t$ or $v_i = u_i$.
    Define $e_i = \hat e_i \cup \{v_i\}$.
    We claim that
    \begin{equation*}
        M'=(M \setminus \{e\}) \cup \{e_i \mid i \in I\}
    \end{equation*}
    satisfies Property~\ref{itm:02-09-lemma-2-property-M-1} and Property~\ref{itm:02-09-lemma-2-property-M-2}.
    Clearly, $M'$ is $t$-shallow.
    Since $|I| \geq 2$ it holds that $M'$ covers more vertices than $M$.
    Denote $U'_i = V_i \setminus V(M')$ to be the set of vertices in $V_i$ that are not covered by $M'$.
    If $i \in I$, then there exist $|I|-1$ edges in $\{e_j \mid j \in I\}$ that cover distinct vertices of $U_i$.
    Moreover, the edge $e_i$ does not cover a vertex of $U_i$.
    Thus, $|U'_i| = |U_i|-(|I|-1) = n_0 - |I| + 1$.
    If $i \notin I$, then all $|I|$ edges in $\{e_j \mid j \in I\}$ cover a vertex in $U_i$.
    But the vertex $v \in e \cap V_i$ is covered exactly once by $M$ (by the definition of $I$) and thus, $v$ is not covered by $M'$.
    Then, $|U'_i| = |U_i| - |I| + 1 = n_0 - |I| + 1$.
    To prove Property~\ref{itm:02-09-lemma-2-property-M-2}, note that $|M'|=|M|+|I|-1$ and $n'_0 := |U'_1| = \cdots = |U'_r| = n_0-|I|+1$.
    Thus, $|M'|=|M|+(n_0 -n'_0) \leq \lceil n/k \rceil (t-1) + n - n_0 + n_0 - n'_0 = \lceil n/k \rceil (t-1) + n - n'_0$.\qedhere
\end{enumerate}
\end{proof}

We can now prove that every $r$-uniform $r$-partite hypergraph $H$ with minimum degree $\delta'_{r-1}(H) \geq n/((r-1)t+1)$ has a $t$-shallow edge set that covers all but at most $r-2$ vertices of each part.
For $t=1$, the result is proven by Theorem~\ref{theorem:02-09-kuhn-matchings}, so we consider the case $t \geq 2$.
As we need these later to deduce Theorem~\ref{theorem:02-09-upper-bound}, we show the following lemma together with some additional properties of the $t$-shallow edge set.
\begin{lem}
\label{lemma:02-09-lemma-3}
Let $r \geq 2$, $t \geq 2$ and $n$ be positive integers and $H=(V_1 \dot\cup \cdots \dot\cup V_r, E)$ be an $r$-uniform $r$-partite hypergraph with parts of size $n$ and minimum degree $\delta'_{r-1}(H) \geq \lceil n/k \rceil$ where $k=(r-1)t+1$.
Then, there exists a $t$-shallow edge set $M$ with the following properties.
Here, let $n_{s,i}$ be the number of vertices $v \in V_i$ with $\degree_M(v)=s$.
\begin{enumerate}[label = (\arabic*)]
    \item\label{itm:02-09-lemma-3-property-M-1} There exists a non-negative integer $n_0$ with $n_{0,1} = n_{0,2} = \cdots = n_{0,r} = n_0$.
    \item\label{itm:02-09-lemma-3-property-M-2} $|M| \leq \left\lceil n/k \right\rceil (t-1) + n - n_0$.
    \item\label{itm:02-09-lemma-3-property-M-3} $M$ covers all but at most $r-2$ vertices of each part of $H$, i.e., $n_0 \leq r-2$.
\end{enumerate}
\end{lem}
\begin{proof}
Let $M$ be a $t$-shallow edge set that satisfies Property~\ref{itm:02-09-lemma-3-property-M-1} and Property~\ref{itm:02-09-lemma-3-property-M-2} and maximizes the number of covered vertices with respect to these both properties.
Assume, for contradiction, that $n_0 \geq r-1$.
Then, we can apply Lemma~\ref{lemma:02-09-lemma-2}.
By Result~\ref{itm:02-09-lemma-2-res-3} of Lemma~\ref{lemma:02-09-lemma-2}, we have $|M|=\lceil n/k \rceil (t-1)+n-n_0$.
By Result~\ref{itm:02-09-lemma-2-res-2} of Lemma~\ref{lemma:02-09-lemma-2}, every edge $e \in M$ has at most one vertex $v \in e$ with $\degree_M(v) \geq 2$.
Then, by Result~\ref{itm:02-09-lemma-1-res-4} of Lemma~\ref{lemma:02-09-lemma-1}, we have $n_{t,i}=\lceil n/k \rceil$ for all $i=1,2,\dots,r$.
We count the number of vertices in the part $V_1$.
Note that $n_{t,1}=\lceil n/k \rceil$.
Since $n_{t,i}=\lceil n/k \rceil$ for all $i=2,3,\dots,r$ and every edge in $M$ has at most one vertex $v$ with $\degree_M(v) \geq 2$, we have $n_{1,1} \geq t (r-1) \lceil n/k \rceil$.
Thus,
\begin{equation*}
    n \geq n_0 + n_{1,1} + n_{t,1}
    \geq n_0 + t (r-1) \left\lceil \frac{n}{k} \right\rceil + \left\lceil \frac{n}{k} \right\rceil
    \geq n_0 + k \left\lceil \frac{n}{k} \right\rceil
    \geq n_0 + n~,
\end{equation*}
a contradiction since we assumed $n_0 \geq r-1 > 0$.
\end{proof}

We are now able to prove that every $r$-uniform $r$-partite hypergraph with parts of size $n$ and minimum degree $\delta'_{r-1}(H) \geq 1 + n/((r-1)t+1)$ has a $t$-shallow hitting edge set.
This condition is tight up to the additive constant $1$, as shown in Theorem~\ref{theorem:02-09-lower-bound}.
\begin{thm}
\label{theorem:02-09-upper-bound}
Let $r \geq 2$ and $t \geq 2$ be positive integers and $H=(V_1 \dot\cup \cdots \dot\cup V_r, E)$ be an $r$-uniform $r$-partite hypergraph with parts of size $n$ and minimum degree
\begin{equation*}
    \delta'_{r-1}(H) \geq \left\lceil \frac{n}{(r-1)t+1} \right\rceil + 1~.
\end{equation*}
Then, there exists a $t$-shallow hitting edge set in $H$.
\end{thm}
\begin{proof}
Let $k=(r-1)t+1$.
By Lemma~\ref{lemma:02-09-lemma-3}, there exists a $t$-shallow edge set $M$ that satisfies Property~\ref{itm:02-09-lemma-3-property-M-1}, Property~\ref{itm:02-09-lemma-3-property-M-2} and Property~\ref{itm:02-09-lemma-3-property-M-3} of Lemma~\ref{lemma:02-09-lemma-3}.
Let $M$ be such a $t$-shallow edge set that maximizes the number of covered vertices with respect to all three properties.
Let $U_i = V_i \setminus V(M)$ be the set of uncovered vertices in $V_i$.
Then, $|U_1|=\cdots=|U_r|=n_0$ for some non-negative integer $n_0 \leq r-2$, by Property~\ref{itm:02-09-lemma-3-property-M-1} and Property~\ref{itm:02-09-lemma-3-property-M-3}.
Moreover, $|M| \leq \lceil n/k \rceil (t-1) + n-n_0$ by Property~\ref{itm:02-09-lemma-3-property-M-2}.
Let $n_{s,i}$ be the number of vertices $v \in V_i$ with $\degree_M(v)=s$.
By Lemma~\ref{lemma:02-09-lemma-1} Result~\ref{itm:02-09-lemma-1-res-1}, we have $n_{t,i} \leq \lceil n/k \rceil$ for all $i=1,2,\dots,r$.
For $n_0 = 0$, there is nothing to show.
Therefore, assume $n_0 > 0$.
We construct a $t$-shallow hitting edge set $M'$ by adding edges to $M$.

For each $i=1,2,\dots,n_0+1$, we build a legal $(r-1)$-set $\hat e_i$ of vertices of $U \setminus U_i$, such that each vertex in $U_1\cup U_2 \cup \cdots \cup U_{n_0 + 1}$ is contained in exactly one $\hat e_i$ and each vertex in $\hat U_{n_0 + 2} \cup \cdots \cup U_{r}$ is contained in at least one and at most two $\hat e_i$'s.
Note that $N_{r-1}(\hat e_i) \subseteq V_i \setminus U_i$, since otherwise we can build a $t$-shallow edge set $M'$ that covers more vertices and satisfies all three properties of Lemma~\ref{lemma:02-09-lemma-3}, a contradiction.
Since $\delta'_{r-1}(H) \geq \lceil n/k \rceil + 1$ and $n_{t,i} \leq \lceil n/k \rceil$ for all $i=1,2,\dots,r$, there exists a vertex $v_i \in N_{r-1}(\hat e_i)$ that is covered at least once but less than $t$ times by $M$.
We define $e_i = \hat e_i \cup \{v_i\}$.
Let $M' = M \cup \{e_i \mid i=1,2,\dots,n_0 + 1\}$.
Then, $M'$ is $t$-shallow and covers all vertices of $H$.
\end{proof}

\section{Conclusions}
\label{sec:conclusions}

Open problems include all cases in which our bounds are not tight yet.
The most intriguing is to determine the least integer $t = t(r)$ for which every regular $r$-uniform hypergraph admits a $t$-shallow hitting edge set.
Our results give that $t(r) \geq \lfloor \log_2 (r+1) \rfloor$ and $t(r) \leq \euler r(1+o(1))$, where we can improve the upper bound to $c\ln r (1+o(1))$ for $c \approx 4.31$ (i.e., nearly matching the lower bound) for the special case of regular $r$-uniform hypergraphs of girth at least four.

Let us also mention that in the first author's Master's thesis~\cite{Planken-msc} it is shown that it is NP-complete to determine whether an $r$-uniform $r$-partite regular hypergraph has a $t$-shallow hitting edge set, provided $t \leq \frac12 \log_2(r) - O(\log \log r)$, i.e., only slightly below the smallest value for $t$ for which such $t$-shallow hitting edge sets \emph{always} exist.

\bibliographystyle{plainurl}
\bibliography{literatur}

\end{document}